\documentclass[10pt]{amsart}

\usepackage{pdfsync, a4,fullpage, comment, enumerate, graphicx, 
latexsym, amsfonts, amsthm, amsmath, amssymb, amscd, mathrsfs, 
stmaryrd, amsopn, eepic, epic, eucal, nicefrac, wasysym, epsfig, 
psfrag, fancyhdr, enumitem}
\usepackage[unicode,breaklinks=true,colorlinks=true]{hyperref}
\usepackage[active]{srcltx}
\usepackage[utf8]{inputenc}
\usepackage[all]{xy}  % commutative diagrams

\newtheorem{theorem}{Theorem}[section]
\newtheorem{lettertheorem}{Theorem}
\newtheorem{prop}[theorem]{Proposition}

\newtheorem{lemma}[theorem]{Lemma}

\newtheorem{cor}[theorem]{Corollary}

\theoremstyle{definition}
\newtheorem{definition}[theorem]{Definition}
\def\mathclap#1{\text{\hbox to 0pt{\hss$\mathsurround=0pt#1$\hss}}}

\newcommand{\La}{\Lambda} 
\newcommand{\la}{\lambda}

\newcommand{\Om}{\Omega}

\newcommand{\om}{\omega}

\newcommand{\ep}{\epsilon} 
\newcommand{\De}{\Delta}
\newcommand{\de}{\delta}
\newcommand{\si}{\sigma} 

\newcommand{\R}{{\mathbb R}}
\newcommand{\Z}{{\mathbb Z}}
\newcommand{\Q}{{\mathbb Q}}
\newcommand{\N}{{\mathbb N}}

\newcommand{\norm}[1]{\| #1\|}
\newcommand{\abs}[1]{\left|#1\right|}
\newcommand{\deriv}[2]{\frac{\partial #1}{\partial #2}}

\renewcommand{\geq}{\geqslant}
\renewcommand{\leq}{\leqslant}

\newenvironment{remark}{\refstepcounter{theorem}\par\medskip\noindent{\bf
Remark~\thetheorem.}}{\unskip\nobreak\par\bigskip}

\newenvironment{example}{\refstepcounter{theorem}\par\medskip\noindent{\bf
Example~\thetheorem.}}{\unskip\nobreak\hfill\hbox{ $\oslash$}\par\bigskip}

\newenvironment{notn}{\refstepcounter{theorem}\par\medskip\noindent{\bf
Notation~\thetheorem.}}{\par\smallskip}

\newcommand{\D}{d_{\mf,[\vec{k}]}^{\nu,\bn}}
\newcommand{\Dp}{d_{\mf,[\vec{k}]}^{p,\nu,\bn}}
\newcommand{\Did}{d_{\mf,\vec{k}}^{\text{Id},\nu,\bn}}
\newcommand{\Dnomk}{d^{\nu,\bn}}
\newcommand{\Didnomk}{d^{\text{Id},\nu,\bn}}
\newcommand{\Ftoric}{F_{\text{toric}}}
\newcommand{\metricspace}{(\semitoric, \Dst )}
\newcommand{\metricspacemmk}{(\mmk,\D)}

\newcommand{\m}{\mathcal{M}}
\newcommand{\mmk}{\mathcal{M}_{\mf,[\vec{k}]}}
\newcommand{\mmknoperm}{\mathcal{M}_{\mf,\vec{k}}}
\newcommand{\mcompl}{\widetilde{\mathcal{M}}}
\newcommand{\mcomplmk}{\widetilde{\mathcal{M}}_{\mf,[\vec{k}]}}
\newcommand{\mf}{{m_f}}
\newcommand{\mele}{m} %an element of \m
\newcommand{\melep}{m'}
\newcommand{\melepp}{m''}
\newcommand{\melen}{m_n}
\newcommand{\st}{\tau} %a single semitoric system

\newcommand{\semitoric}{\mathcal{T}} % the space of all semitoric systems
\newcommand{\semitoriccpt}{\mathcal{T}^{\text{cpt}}} % the space of all compact semitoric systems
\newcommand{\semitoricm}{\mathcal{T}_\mf} %space of semitoric systems of complexity mf
\newcommand{\semitoricmk}{\mathcal{T}_{\mf,\vec{k}}} %space of semitoric systems of complexity mf and k values k
\newcommand{\semitoricmkperm}{\mathcal{T}_{\mf,[\vec{k}]}} %space of semitoric systems of complexity mf and k values k up to permutation
\newcommand{\semitoricmkc}{\mathcal{T}_{\mf,(k_j+c)_{j=1}^\mf}} %space of semitoric systems of complexity mf and k values k+c

\newcommand{\semitoricz}{\semitoric_0}
\newcommand{\semitoricJC}{\semitoric_1}
\newcommand{\toric}{\mathcal{T}_\mathbb{T}} %toric systems are semitoric ones with complexity 0

\newcommand{\perm}{\mathcal{S}^\mf}
\newcommand{\permkkp}{\mathcal{S}^\mf_{\vec{k},\vec{k}'}}

\newcommand{\semitoricmkprime}{\mathcal{T}_{\mf,\vec{k}'}} %space of semitoric systems of complexity mf and k values k'
\newcommand{\permkkpp}{\mathcal{S}^\mf_{\vec{k},\vec{k}''}}
\newcommand{\permkppkp}{\mathcal{S}^\mf_{\vec{k}'',\vec{k}'}}

\newcommand{\Dplain}{d}
\newcommand{\Dpplain}{d^p}
\newcommand{\Didplain}{d^{\text{Id}}}
\newcommand{\Dqplain}{d^q}
\newcommand{\Dpqplain}{d^{p\circ q^{-1}}}

\newcommand{\stfull}{\big( [\dew], (h_j)_{j=1}^\mf, ((S_j)^\infty)_{j=1}^\mf \big)}
\newcommand{\stfullp}{\big( [\dewp], (h_j')_{j=1}^\mf, ((S_j')^\infty)_{j=1}^\mf \big)}
\newcommand{\stfulla}{\big( [\Aw], (h_j)_{j=1}^\mf, ((S_j)^\infty)_{j=1}^\mf \big)}
\newcommand{\stfullap}{\big( [\Awp], (h_j')_{j=1}^\mf, ((S_j')^\infty)_{j=1}^\mf \big)}
\newcommand{\Aw}{A_\text{w}}
\newcommand{\Awp}{A_\text{w}'}

\newcommand{\ts}{\sum_{i,j\geq0} \si_{i,j}X^i Y^j}
\newcommand{\tsp}{\sum_{i,j\geq0} \si_{i,j}'X^i Y^j}

\newcommand{\tsz}{\sum_{i,j\geq0} \si_{i,j}^0 X^i Y^j}
\newcommand{\bn}{{\{b_n\}_{n=0}^\infty}}
\newcommand{\bnp}{{\{b_n'\}_{n=0}^\infty}}
\newcommand{\rxy}{\R[[X,Y]]} %space of Taylor Series
\newcommand{\rxyz}{\R[[X,Y]]_0} %space of Taylor Series with no const and sig_2 \in [0,2\pi)
\newcommand{\dts}{d^\bn_{\rxy}}
\newcommand{\dtsz}{d^\bn_0}

\newcommand{\symdiff}{\ominus}
\newcommand{\poly}{\text{Polyg}(\R^2)}%{\mathcal{P}}
\newcommand{\lwpoly}{\mathcal{LW}\text{Polyg}_\mf(\R^2)}%{\mathcal{LWP}_\mf}
\newcommand{\dpoly}{\mathcal{D}\text{Polyg}_{\mf,[\vec{k}]}(\R^2)}%{\mathcal{DP}_{\mf,[\vec{k}]}}
\newcommand{\dpolynok}{\mathcal{D}\text{Polyg}_{\mf}(\R^2)}%{\mathcal{DP}_{\mf}}
\newcommand{\dpolynokz}{\mathcal{D}\text{Polyg}_{0}(\R^2)}%{\mathcal{DP}_0}
\newcommand{\distpoly}{d_{\text{P}}^\nu}%used to be \mathcal{P}
\newcommand{\distpolyp}{d_{\text{Polyg}}^{p,\nu}}
\newcommand{\distpolyid}{d_{\text{P}}^{\text{Id},\nu}}
\newcommand{\distpolyidplain}{d_{\text{P}}^{\text{Id}}}
\newcommand{\lwp}{(\De, (\ell_{\la_j}, \ep_j, k_j)_{j=1}^\mf)}
\newcommand{\lwpkp}{(\De, (\ell_{\la_j}, \ep_j, k_j')_{j=1}^\mf)}
\newcommand{\lwpp}{(\De', (\ell_{\la_j'}, \ep_j', k_j')_{j=1}^\mf)}
\newcommand{\lwpplus}{(\De, (\ell_{\la_j}, +1, k_j)_{j=1}^\mf)}
\newcommand{\lwpplusone}{(\De^1, (\ell_{\la_j}^1, +1, k_j^1)_{j=1}^\mf)}
\newcommand{\lwpplustwo}{(\De^2, (\ell_{\la_j}^2, +1, k_j^2)_{j=1}^\mf)}
\newcommand{\dew}{\De_\text{w}}
\newcommand{\dewp}{\De_\text{w}'}
\newcommand{\Gmf}{G_\mf}

\newcommand{\tk}{\mathcal{G}}

\newcommand{\vertr}{\text{Vert}(\R^2)}

\newcommand{\dmu}{\,\mathrm{d}\mu}
\newcommand{\dnudmu}{\frac{\mathrm{d}\nu}{\mathrm{d}\mu}}

\newcommand{\lwppluskp}{(\De, (\ell_{\la_j}, +1, k_j')_{j=1}^\mf)}
\newcommand{\apoly}{\mathcal{P'}_{\mf,[\vec{k}]}}
\newcommand{\bpoly}{\mathcal{P''}_{\mf,[\vec{k}]}}
\newcommand{\cpoly}{\mathcal{P'''}_{\mf,[\vec{k}]}}
\newcommand{\mm}{\mathcal{M}_{\mf}}
\newcommand{\mmka}{\mathcal{M'}_{\mf,[\vec{k}]}}
\newcommand{\mmkb}{\mathcal{M''}_{\mf,[\vec{k}]}}
\newcommand{\mmkc}{\mathcal{M'''}_{\mf,[\vec{k}]}}
\newcommand{\lwpplusa}{(A, (\ell_{\la_j}, +1, k_j)_{j=1}^\mf)}
\newcommand{\lwpplusaeq}{([A], (\ell_{\la_j}, +1, k_j)_{j=1}^\mf)}
\newcommand{\lwpplusakp}{(A, (\ell_{\la_j}, +1, k_j')_{j=1}^\mf)}
\newcommand{\lwpplusb}{(B, (\ell_{\la_j}, +1, k_j)_{j=1}^\mf)}
\newcommand{\polycompl}{\widetilde{\mathcal{D}\text{Polyg}}_{\mf,[\vec{k}]}(\R^2)}
\newcommand{\akep}{A^k_{\vec{\ep}}}
\newcommand{\lwppplus}{(\De', (\ell_{\la_j}, +1, k_j)_{j=1}^\mf)}

\newcommand{\Dst}{\mathcal{D}^{\nu,\bn}}

\newcommand{\bigo}{\mathcal{O}}

\title{Moduli spaces of semitoric systems}
\author{Joseph Palmer}

\begin{document}

\begin{abstract}
 Recently Pelayo-V\~{u} Ng\d{o}c classified simple semitoric integrable 
 systems in terms of five symplectic invariants. Using this 
 classification we define a family of metrics on the space of 
 semitoric integrable systems.
 The resulting metric space is incomplete and we construct the completion.
\end{abstract}

\maketitle

\section{Introduction}\label{sec_intro}

Toric integrable systems are classified by the image of their momentum map,
which is a Delzant polytope.
In~\cite{PePRS2013} Pelayo-Pires-Ratiu-Sabatini
define a metric on the space of Delzant polytopes via the volume of the 
symmetric difference and pull this back to produce a metric on the 
moduli space of toric integrable systems.  
The construction of this metric is related to the Duistermaat-Heckman measure~\cite{DuHe1982}.

In~\cite{PeVNsemitoricinvt2009, PeVNconstruct2011}, Pelayo and V\~{u} Ng\d{o}c provide a 
complete classification for a broader class of integrable systems, those known 
as semitoric, in terms of a collection of several invariants.  
 A \emph{semitoric integrable system}~\cite{PeVNsemitoricinvt2009}
 is a 4-dimensional, connected, symplectic 
 manifold $(M,\om)$ with a momentum map $F=(J,H):M\to\R^2$ such that:
 \begin{enumerate}[noitemsep]
  \item the function $J$ is a proper momentum map for a Hamiltonian circle action on $M$;
  \item $F$ has only non-degenerate singularities (as in Williamson~\cite{Wi1996})
        without real-hyperbolic blocks.\label{cond_nondeg}
 \end{enumerate}

Notice that though semitoric systems are required to be 4-dimensional
there is much more freedom in the choice of momentum map compared to toric systems and $M$ 
is not required to be compact (the non-compact toric case is treated by Karshon-Lerman~\cite{KaLe2014}).
Condition \eqref{cond_nondeg} implies that if $p\in M$ 
is a critical point of $F$ then there exists some $2\times2$ matrix $B$ such that 
$\widetilde{F} = B\circ(F-F(p))$ is given by one of three standard forms.
By Eliasson~\cite{El1984,El1990}
there exists a local symplectic chart $(x,y,\eta,\xi)$ centered at $p$ 
which puts $\widetilde{F}$ into one of the three possible singularity types:
\begin{enumerate}[nosep]
 \item transversally elliptic singularity: $\widetilde{F}(x,y,\eta,\xi) = (\eta + \bigo (\eta^2), \frac{x^2+\xi^2}{2}) + \bigo((x,\xi)^3)$;
 \item elliptic-elliptic singularity: $\widetilde{F}(x,y,\eta,\xi) = (\frac{x^2 + \xi^2}{2},\frac{y^2+\eta^2}{2}) + \bigo((x,\xi,y,\eta)^3)$;
 \item focus-focus singularity: $\widetilde{F}(x,y,\eta,\xi) = (x\xi + y\eta, x\eta - y\xi) + \bigo((x,\xi,y,\eta)^3)$.
\end{enumerate}

A semitoric integrable system $(M,\om,F=(J,H))$ is said to be 
\emph{simple} if there is at most one focus-focus critical point in 
$J^{-1}(x)$ for all $x\in\R$.
A similar (but weaker)
assumption is generic according to Zung~\cite{Zu1996}, that each fiber 
$F^{-1}(c)$ for $c\in\R^2$ contains at most one critical point $p\in M$.
Any semitoric system has only finitely many focus-focus critical points (See V\~{u} Ng\d{o}c~\cite{VN2007})
so we will denote them by $m_1,\ldots,m_\mf\in M$ and the associated
singular values are denoted $c_j = F(m_j)$, $j=1, \ldots, \mf$. 
All semitoric systems studied in this article are assumed to be simple
and we label them such that $J(m_1)<\ldots<J(m_\mf)$.
Suppose that $(M_i, \om_i, F_i=(J_i,H_i))$ is a semitoric system
for $i=1,2$.  An \emph{isomorphism of semitoric systems} is 
a symplectomorphism $\phi:M_1 \to M_2$ such that 
$\phi^* (J_2, H_2) = (J_1, f(J_1,H_1))$ where $f:\R^2\to\R$ is a smooth 
function such that $\deriv{f}{H_1}$ nowhere vanishes (it is either always 
strictly positive
or always strictly negative).  We denote by 
$\semitoric$ the space of simple semitoric systems modulo isomorphism.

The goal of this paper is to define a metric on the space of invariants and 
thus induce a metric on $\semitoric$, thereby addressing Problem 2.43 
from Pelayo-V\~{u} Ng\d{o}c~\cite{PeVNfirst2012}, in which the authors ask for a description of
the topology of the moduli space of semitoric systems.  Problems 2.44 and
2.45 in the same article are related to the closure of $\semitoric$ in the
moduli space of all integrable systems, so in this paper we also compute the
completion of the space of invariants, which corresponds to the completion 
of $\semitoric$, in order to lay the foundation to begin work on these problems. 
The main result of this paper, Theorem~\ref{thm_main}, states that the function
we propose is a metric on $\semitoric$ and describes the completion of the space
of invariants.  Theorem~\ref{thm_main} is stated in Section~\ref{sec_mainthm}
after we have defined the metric.

%\pagebreak
\subsection{Notation Index}
Here we list some of the notation used in this article:
\vspace{5pt}
\begin{center}
\begin{tabular}{ll}
 $\semitoric$              &    Moduli space of simple semitoric systems, Section~\ref{sec_intro}\\
 $\semitoricm$             &    Elements of $\semitoric$ with $\mf$ focus-focus singular points, Section~\ref{sec_mf}\\
 $\semitoricmk$            &    Elements of $\semitoric$ in twisting index class $\vec{k}\in\Z^\mf$, Definition~\ref{def_semitoricmk} \\
\vspace{8pt}
 $\semitoricmkperm$        &    Elements of $\semitoric$ in generalized twisting index class $\vec{k}\in\Z^\mf$, Definition~\ref{def_semitoricmk} \\
 $\m$                      &    Semitoric lists of ingredients, Definition~\ref{def_listofingre} \\
 $\mm$                     &    Semitoric lists of ingredients with complexity $\mf$, Definition~\ref{def_listofingre} \\
 $\mmknoperm$              &    Elements of $\mm$ in twisting index class $\vec{k}\in\Z^\mf$, Section~\ref{sec_topology}\\
 $\mmk$                    &    Elements of $\mm$ in generalized twisting index class $\vec{k}\in\Z^\mf$, Definition~\ref{def_semitoricmk}\\
\vspace{8pt}
 $\mcompl$                 &    The completion of $\m$, Definition~\ref{def_mcomplmk}\\
 $\poly$                   &    Rational convex polygons in $\R^2$, Section~\ref{sec_affinetwist}\\
 $\lwpoly$                 &    Labeled weighted polygons of complexity $\mf$, Definition~\ref{def_lwpoly} \\
 $\lwp$                    &    Typical element of $\lwpoly$, Definition~\ref{def_lwpoly}\\
\vspace{8pt}  
 $\dpolynok$               &    Labeled Delzant semitoric polygons of complexity $\mf$, Definition~\ref{def_delzsemi} \\
 $\Dst$                    &    Metric on $\semitoric$, Definition~\ref{def_fullmetric}\\
 $\Dnomk$                  &    Metric on $\m$, Definition~\ref{def_fullmetric}\\
 $\D$                      &    Metric on on $\mmk$, Definition~\ref{def_metric}\\
\vspace{8pt}
 $\Dp$                     &    Comparison with alignment $p\in\perm$, Definition~\ref{def_metric}\\
 $\Gmf\times\tk$           &    The group $\{-1,+1\}^\mf\times\{T^k \mid k\in\Z\}$, Section~\ref{sec_ggaction}\\
 $\bn$                     &    Linear summable sequence, Definition~\ref{def_metricts}\\
 $\permkkp$                &    Appropriate permutations for $\vec{k}, \vec{k'}\in\Z^\mf$, Definition~\ref{def_permkkp}

\end{tabular}
\end{center}

\section{Background: The classification of semitoric integrable systems}\label{sec_classthm}
 Since it is necessary for the construction of the metric, in this
 section we describe in detail the five invariants which completely classify
 simple semitoric systems.
 Compact toric integrable systems are classified in terms of Delzant polytopes.  
 In the semitoric case a polygon plays a role but the complete invariant 
 must contain more information.  
 Loosely speaking, the complete invariant of semitoric systems is a collection 
 of convex polygons in $\R^2$ (which may not be compact) each with a finite number 
 of distinguished points corresponding to the focus-focus singularities 
 labeled by a Taylor series and an integer (See Figure \ref{fig_complinvt}).

\subsection{The number of singular points invariant}\label{sec_mf}
 In~\cite[Theorem 1]{VN2007} V\~{u} Ng\d{o}c proves that any (simple or not) semitoric 
 system has finitely many focus-focus singular points.  Thus, to a system
 we may associate an integer $0\leq \mf < \infty$ which is the total 
 number of focus-focus points in the system.  The singular 
 points are preserved by isomorphism so this is an invariant of the system.
  For any nonnegative integer $\mf\in\Z_{\geq0}$ let $\semitoricm$ denote 
  the collection of simple semitoric systems with $\mf$ focus-focus points
  modulo semitoric isomorphism.

 Toric systems, which have no focus-focus singular points, correspond to a 
 proper subset of $\semitoricz$.
 There is some subtlety in this correspondence because of the difference 
 between toric and semitoric isomorphisms, see Section~\ref{sec_toric}, 
 but it can be seen that the topology on toric systems from Pelayo-Pires-Ratiu-Sabatini~\cite{PePRS2013}
 is compatible with the topology defined in the present article.
 This is the content of Corollary~\ref{cor_toric}).
 The elements of $\semitoricJC$ are known as a Jaynes-Cummings type 
 systems, as in Pelayo-Le Floch~\cite{LFPeVN2014}.  An important
 example of a Jaynes-Cummings type system is the coupled spin-oscillator 
 which is defined in~\cite{Cu1965, JaCu1963} and studied in detail by V\~{u} Ng\d{o}c and Pelayo-V\~{u} Ng\d{o}c in~\cite{VN2007, PeVNspin2012}.
 This system has its origins in the study of quantum optics and is a prime
 example of the importance of semitoric systems in physics.

  \begin{figure}
 \centering
 \includegraphics[height=120pt]{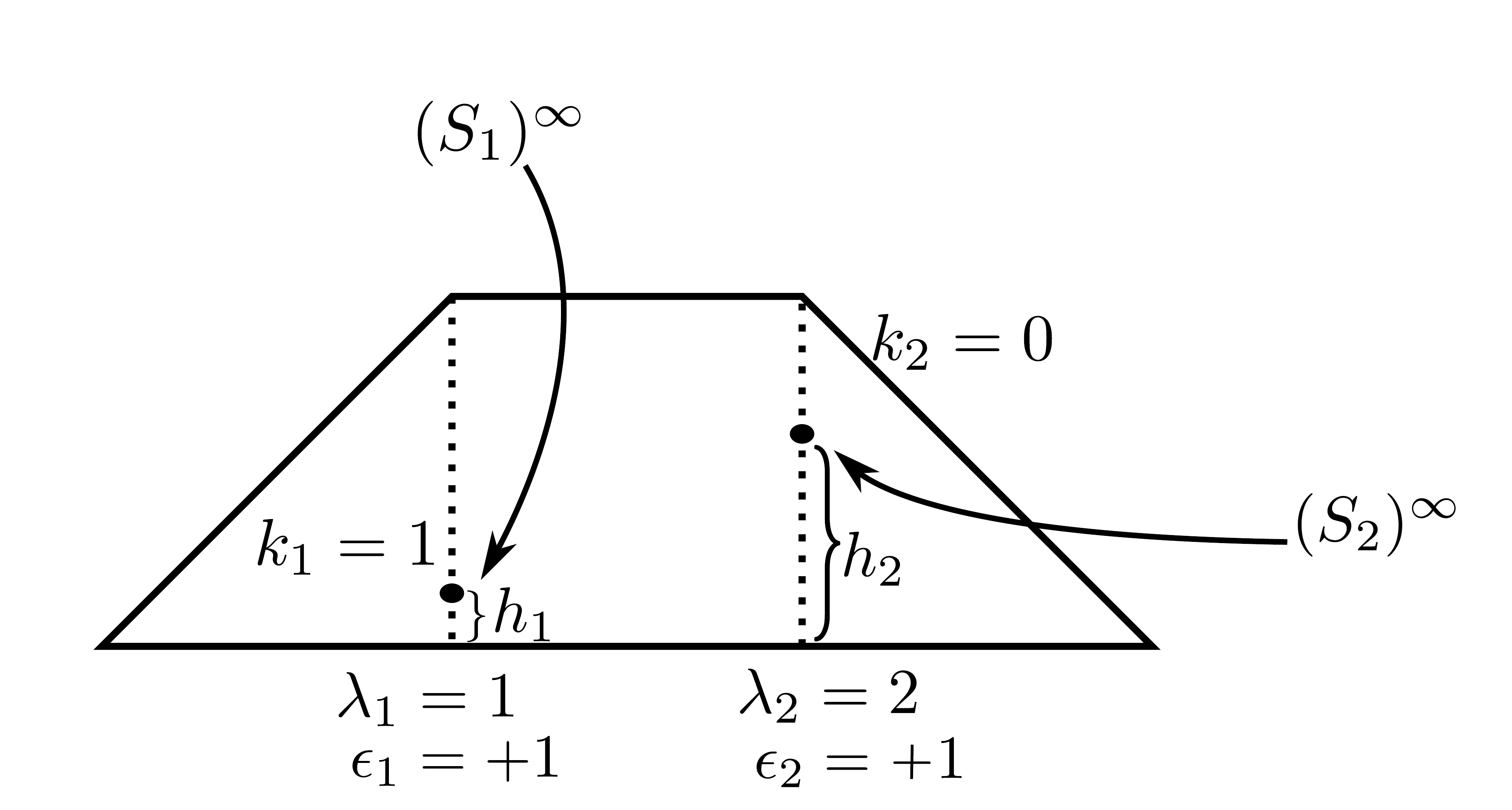}
 \caption{The complete semitoric invariant is a collection of polygons with distinguished
 points $\{c_1, \ldots, c_{\mf}\}$ each labeled with extra information: a Taylor series
 $(S_j)^\infty$, an integer $k_j$ known as the twisting index, and 
 an element $\ep_j\in\{-1,+1\}$ known as the cut direction.  There is one polygon in
 the family for each possible choice of cut directions and each allowed choice of 
 twisting indices.}
 \label{fig_complinvt}
\end{figure}
 
\subsection{The Taylor series invariant}

The next invariant we will study completely classifies the structure of 
a focus-focus critical point in the neighborhood of a fiber up to isomorphism,
originally formulated in V\~{u} Ng\d{o}c~\cite{VN2003}.
It is defined in terms of the length of certain flow lines of the Hamiltonian 
vector fields for the components of the momentum map, and can also be viewed
as the germ of the generating function at the focus-focus point.  The details can be 
found in~\cite{VN2003, PeVNsymplthy2011}.

\begin{definition}
 Let $\rxy$ refer to the algebra of real formal power series in two variables and 
 let $\rxyz\subset \rxy$ be the subspace of series $\ts$ which have $\si_{0,0}=0$ 
 and $\si_{0,1} \in [0,2\pi)$.
\end{definition}

 The Taylor series invariant is one element of $\rxyz$ for each of the $\mf$ focus-focus points.

 \subsection{The affine invariant and the twisting index invariant}\label{sec_affinetwist}
The affine invariant is similar to the polygon from Delzant's result, except in this case we 
instead have a family of polygons related by specific linear transformations.  The twisting 
index describes how each critical point sits with respect to a privileged momentum map.  
These two invariants will be described together because the twisting indices which label 
each critical point will be defined only up to the addition of a common integer related 
to the choice of polygon.

A \emph{convex polygon} is the intersection in $\R^2$ of (finitely or infinitely many) 
 closed half planes such that on each compact subset of $\R^2$ there are at 
 most finitely many corner points.  A convex polygon is rational if each edge 
 is directed along a vector with rational coefficients.
 We denote the set of all rational convex polygons by $\poly$.
For $\la\in\R$ let $\ell_\la = \{(x,y)\in\R^2 \mid x=\la\}$ and let $\vertr=\{\ell_\la \mid \la\in\R\}$.

\begin{definition}\label{def_lwpoly}
 A \emph{labeled weighted polygon of complexity $\mf\in\Z_{\geq0}$} is an element 
 \[
  \De_{\mathrm{weight}}=\lwp\in \poly\times\big(\vertr\times\{-1,+1\}\times\Z\big)^\mf
 \]
 with
 \[
  \min_{s\in\De}\pi_1 (s)<\la_1<\la_2<\ldots<\la_\mf<\max_{s\in\De}\pi_1(s)
 \]
 where $\pi_1:\R^2\to\R$ is the projection onto the $x$-coordinate.
 We denote the space of labeled weighted polygons of complexity $\mf$ by $\lwpoly$,
 and we use the simplified notation $\dew = \De_{\textrm{weight}}$.
\end{definition}

Notice there is a triple $(\ell_{\la_j},\ep_j,k_j)$ associated with the 
singular point $c_j$ for each $j=1,\ldots,\mf$.  These are related to the
critical points of the semitoric system as follows:
$\la_j = \pi_1(c_j)$; $\ep_j = +1$ if the cut at $c_j$ goes up (in the positive $H$ direction) and 
$\ep_j = -1$ if the cut at $c_j$ goes down; and
$k_j$ is the twisting index of $c_j$.

Here we will briefly review how the affine invariant is produced by V\~{u} Ng\d{o}c~\cite{VN2007}.
Let $(M,\om,F = (J,H))$ be a semitoric system.
Consider the set $F(M)\subset\R^2$.  In the toric case this is the Delzant polygon.
Let $c_1,\ldots,c_\mf\in F(M)$ denote the images of the 
focus-focus points and let $B_r = \text{Int}(F(M))\setminus\{c_1,\ldots,c_\mf\}$
which is precisely the regular values of $F$~\cite[Remark 3.2]{PeVNsemitoricinvt2009}.
For each $j=1,\ldots,\mf$ remove from $B_r$ the line segment 
$\ell_{\la_j}^{\ep_j}$ which starts at $c_j$ and goes upwards if 
$\ep_j= 1$ and downwards if $\ep_j=-1$ to form the set $B_r^{\vec{\ep}}$,
where $\vec{\ep} = (\ep_1, \ldots, \ep_\mf)$.
Now, $B_r^{\vec{\ep}}$ is a simply connected set of regular values of $F$ so define
a global toric momentum map 
\[
 \Ftoric:F^{-1}(B_r^{\vec{\ep}})\to\R^2
\]
and define $\De = \overline{\Ftoric(B_r^{\vec{\ep}})}$, the closure.
The polygon produced depends on the choice of 
$(\ep_j)_{j=1}^\mf$ and of the toric momentum map on $B_r^{\vec{\ep}}$.
The distinguished points in each polygon are the image of the 
focus-focus singular points under $\Ftoric$. Of course, we are omitting 
many details in this explanation.  Again, the interested reader should 
see the papers of Pelayo-V\~{u} Ng\d{o}c~\cite{PeVNsemitoricinvt2009,PeVNconstruct2011}.

For $k\in\Z$ let $T^k$ be given by 
\begin{equation}\label{eqn_tk}
T^k = \left(\begin{array}{cc}1 & 0\\k & 1\end{array}\right) \in \text{GL}(2,\Z).
\end{equation}

\begin{definition}\label{def_cornertypes_delzant}
 Let $\De\in\poly$ be a rational convex polygon.  We say that a 
 \emph{vertex} of $\De$ is a point in the boundary $\partial\De$
 where the meeting edges are not co-linear.  A point is said to 
 be in the \emph{top-boundary} of $\De$ if it is the top end of 
 a vertical segment formed by intersecting $\De$ with a vertical line.
 Suppose that $z$ is a vertex of $\De$ and $(u,v)$ are a pair of 
 primitive integral vectors starting at $z$ and extending along 
 the direction of the edges which meet at $z$ in the order
 such that $\mathrm{det}(u,v)\geq 0$.
 Then the point $z$ is called a \emph{corner} and is said to
 \begin{enumerate}
  \item satisfy the \emph{Delzant condition} if $\det(u,v)=1$;
  \item satisfy the \emph{hidden Delzant condition} if it belongs to the top boundary and $\det(u,T^1v)=1$;
  \item satisfy the \emph{fake condition} if it belongs to the top boundary and $\det(u,T^1v)=0$.
 \end{enumerate}
 Notice that it is possible for a corner to satisfy both the
 Delzant and the fake conditions simultaneously.
 A rational convex polygon $\De\in\poly$ is called \emph{Delzant} if 
 it is compact and every corner is a Delzant corner.  This implies that
 it is cut out by finitely many half-planes.
\end{definition}

\subsection{The action of \texorpdfstring{$\Gmf\times\tk$}{the symmetry groups}}\label{sec_ggaction}
 
 In order for isomorphic systems to produce the same invariants, 
 we must consider the collection of invariants we have so far modulo a group action.  
 \begin{notn}
  Throughout this article when referring to an $\mf$-tuple such as $(k_j)_{j=1}^\mf$ or $(\ep_j)_{j=1}^\mf$ for simplicity we will sometimes use vector notation.  That is, we may refer to these $\mf$-tuples as $\vec{k}$ and $\vec{\ep}$, respectively.  These vectors will always have length $\mf$.
 \end{notn} 
 Let $\Gmf = \{-1,+1\}^{\mf}$ and $\tk = \{T^k\mid k\in\Z\}$ where 
 $T^k$ is as in Equation \eqref{eqn_tk}.  
 Given $\ell\in\vertr$ where $\ell = \pi_1^{-1}(\la)$ for some $\la\in\R$ define
 $t_{\ell}^k:\R^2\to\R^2$ by
 \[
  t_{\ell}^k (x,y) = \left\{ \begin{array}{ll}(x,y) & , x\leq \la\\ (x, k(x-\la)+y) &, x>\la. \end{array}\right.
 \]
 That is, $t_{\ell}^k$ acts as the identity on the left of $\ell$ and,
 after a translation of coordinates which moves the origin onto $\ell$,
 acts as $T^k$ to the right of $\ell$.  For $\vec{u}\in\Z^\mf$ let
 $t_{\vec{u}}=t^{u_1}_{\ell_1} \circ \cdots \circ t^{u_{\mf}}_{\ell_\mf}$
 where $\ell_j = \ell_{\la_j}$.
 We define the action of $\Gmf\times\tk$ on $\lwpoly$ by
\begin{equation}\label{eqn_ggaction}
 ((\ep_j')_{j=1}^\mf, T^{k})\cdot\lwp = (t_{\vec{u}}(T^k\De),\left(\ell_{\la_j}, \ep_j'\ep_j, k+k_j\right)_{j=0}^\mf)
\end{equation}
where $\vec{u}=\left( \nicefrac{(\ep_j-\ep_j\ep_j')}{2}\right)_{j=1}^\mf.$
\begin{remark}\label{rmk_e}
 Notice that if $\lwp$ is changed via the action of $\Gmf$ to have 
 $\ep'_j \in \{-1,1\}$ instead of $\ep_j$ for each $j=1,\ldots,\mf$
 then the new polygon is $t_{\vec{u}}(\De)$ where 
 $\vec{u}=(\nicefrac{(\ep_j-\ep'_j)}{2})_{j=1}^\mf\in\{-1,0,1\}^\mf$.
 Thus, the orbit of $\De$ under the action of $\Gmf$ may be written 
 as $( t_{\vec{u}}(\De))_{\vec{u}\in\{0,1\}^\mf}$ if $\De$ is the 
 polygon with $\ep_j=+1$ for all $j=1,\ldots,\mf$.
\end{remark}
The orbit under this action is the appropriate invariant.
The choice of cut direction and constant by which to shift the 
twisting indices parameterize the collection of all polygons in a given orbit.
Notice that the action of $t_{\vec{u}}$ does not necessarily preserve convexity,
but it will in the case of the polygons we are interested in (Proposition \ref{prop_convex}).
\begin{definition}\label{def_delzsemi}
 A \emph{labeled Delzant semitoric polygon} is the equivalence class 
 \[
  [\dew]\in\lwpoly / (\Gmf\times\tk)
 \]
 of an element $\dew=\lwpplus$ satisfying the following.
 \begin{enumerate}
  \item The intersection of $\De$ and any vertical line is either compact or empty;\label{part_intvert}
  \item each $\ell_{\la_j}$ intersects the top boundary of $\De$;
  \item \label{item_three} each point in the top boundary which is also in some $\ell_{\la_j}$ satisfies either the hidden or fake
   condition;
  \item all other corners satisfy the Delzant condition.
 \end{enumerate}
 The corners from item~\eqref{item_three} which are in the intersection of the top boundary of $\De$ and one of the $\ell_{\la_j}$
 which satisfy the hidden or fake condition
 are known as \emph{hidden} and \emph{fake corners}, respectively, and the other corners are
 known as \emph{Delzant corners}.
 The space of labeled Delzant semitoric polygons is denoted by \[\dpolynok \subset \lwpoly / (\Gmf\times\tk).\]
\end{definition}

Notice that while it is possible for a corner to satisfy both the fake and Delzant conditions,
the fake corners can be distinguished from the Delzant corners in a labeled Delzant
semitoric polygon because the fake corners are in the intersection of the top boundary
of $\De$ with some $\ell_{\la_j}$ and the Delzant corners are not.

\begin{remark}
 For each focus-focus point $c_j$ the integer $k_j$ depends on the choice
 of representative of $[\dew]$, but given any two focus-focus points
 $c_j$ and $c_i$ the difference of the associated integers $k_j - k_i$
 is preserved under the action of $\Gmf\times\tk$ so it is the same for any choice of representative.
\end{remark}

Any set satisfying Condition~\eqref{part_intvert} of Definition~\ref{def_delzsemi}
is said to have \emph{everywhere finite height}.
The following Proposition is a restatement of~\cite[Lemma 4.2]{PeVNconstruct2011}. Since a 
preferred representative $\De$ can be chosen with $\vec{\ep}=(+1,\cdots,+1)$ we see that 
the orbit of $\De$ under $\Gmf$ is a subset of $\poly$.
\begin{prop}\label{prop_convex}
 Suppose $\dew\in\lwpoly$ satisfies items (1)-(4) in Definition \ref{def_delzsemi} and 
 $\dew=[\lwpplus]$. Then for each $\vec{u}\in\{0,1\}^\mf$ the set $t_{\vec{u}}(\dew)$ is convex.
\end{prop}

\subsection{The volume invariant}\label{sec_h}
 The action of $\Gmf\times\tk$ can change the vertical position of the 
 images of the focus-focus points, but their height with respect to the
 bottom of the polygon is preserved.
 
 \begin{definition}\label{def_h}
 Suppose $[\dew]\in \dpolynok$ with associated toric momentum map $F$.
 For $j=1,\cdots, \mf$ we define $0<h_j<\text{length}(\dew\cap\ell_{\la_j})$
 by 
 \[
  h_j = F(m_j)-\min_{s\in\De\cap\ell_{\la_j}} \{\pi_2 (s)\}
 \]
 where $\pi_2:\R^2\to\R^2$ is the projection onto the second coordinate and
 $\lwp\in[\dew]$ is any representative.
 \end{definition}
 This is well defined for any choice of polygon in the same equivalence 
 class by~\cite[Lemma 5.1]{PeVNsemitoricinvt2009}. 
 The word ``volume'' is used here because $h_j$ can also be viewed as the
 volume of a specific submanifold of $M$~\cite{PeVNsemitoricinvt2009}.

 \subsection{The classification theorem}
 Now that we have defined all of the invariants we can state the result of 
 Pelayo-V\~{u} Ng\d{o}c found in~\cite{PeVNsemitoricinvt2009, PeVNconstruct2011}.
 \begin{definition}[Pelayo-V\~{u} Ng\d{o}c~\cite{PeVNconstruct2011}]\label{def_listofingre}
  A \emph{semitoric list of ingredients} is
  \begin{enumerate}
   \item \label{ingre_mf} a nonnegative integer $\mf$;
   \item \label{ingre_poly} a labeled Delzant semitoric polygon $[\dew]=[\lwp]$ of complexity $\mf$;
   \item \label{ingre_hj} a collection of $\mf$ real numbers $h_1, \ldots, h_\mf \in \R$ such that 
         $0<h_j<\text{length}(\pi_2(\De\cap\ell_{\la_j}))$ for each $j=1,\ldots,\mf$; and
   \item a collection of $\mf$ Taylor series $(S_1)^\infty, \ldots, (S_{\mf})^\infty\in\rxyz$.
  \end{enumerate}
  In other words, a semitoric list of ingredients is a nonnegative integer $\mf$ and an 
  element of $\dpolynok \times \R^\mf \times \rxyz^\mf$ where the $j^{\text{th}}$ element of
  $\R$ must be in the interval $(0, \text{length}(\pi_2(\De\cap\ell_{\la_j})))$. 
  Let $\m$ denote the collection of all semitoric lists of ingredients and let 
  $\mm$ be lists of ingredients with Ingredient \eqref{ingre_mf} equal to the 
  nonnegative integer $\mf$.
 \end{definition}
 
 Notice how the ingredients interact in Definition~\ref{def_listofingre}.  
 Ingredient~\eqref{ingre_mf} determines the number of copies of each
 ingredient associated to the focus-focus points (the triple $(\ell_{\la_j}, \ep_j, k_j)$ and the real number $h_j$) and Ingredient~\eqref{ingre_hj} is in an interval determined by
 Ingredient~\eqref{ingre_poly}.
 
 %\begin{remark}
 % We can write $\mm$ as a Cartesian product instead of a subset of a Cartesian product
 % by replacing $h_j \in (0, \text{length}(\pi_2(\ell_{\De\cap\la_j})))$ with 
 % \[
 %  \widetilde{h_j} = \frac{h_j}{\text{length}(\pi_2(\De\cap\ell_{\la_j}))}\in (0,1)
 % \]
 % for $j=1, \ldots, \mf$. Then we have 
 % \[
 %  \mm \cong \dpolynok \times (0,1)^\mf \times \rxyz^\mf.
 % \]
 %\end{remark}

 \begin{theorem}[Pelayo-V\~{u} Ng\d{o}c~{\cite[Theorem 4.6]{PeVNconstruct2011}}]\label{thm_class}
  There exists a bijection between the set of simple semitoric integrable 
  systems modulo semitoric isomorphism and $\m$, the set of semitoric lists
  of ingredients.  In particular, the mapping
\[
  \Phi\colon  \semitoric     \to     \m
  \]
  which sends the isomorphism class of semitoric systems 
  $[(M,\om,(H,J))]$ to the collection of associated semitoric
  ingredients $\big([\lwp], (h_j)_{j=1}^\mf, ((S_j)^\infty)_{j=1}^\mf\big)$,
  as described above,
 is a bijection.
 \end{theorem}
 
\section{Construction of metric and statement of main theorem}\label{sec_mainthm}
 To define a metric on $\semitoric$ we will first define a metric on each 
 invariant and then we will combine all of these metrics to form a metric
 on $\m$.  Finally, we will pull this metric back by the map in 
 Theorem~\ref{thm_class} to produce a metric on the space of semitoric systems.
 This is the same strategy used by Pelayo-Pires-Ratiu-Sabatini in~\cite{PePRS2013}.
\subsection{Comparing the Taylor series invariant}\label{sec_comparets}
 First we will define a metric on the Taylor series invariant. For $\ts\in\rxyz$
 note that the term $\si_{0,1}$ should
 actually be regarded as an element of $S^1 = \R/2\pi\Z$. This can be seen from the
 construction in V\~{u} Ng\d{o}c~\cite{VN2003}.

\begin{definition}\label{def_metricts}
 Suppose that $\bn$ is a sequence such that $b_n\in(0,\infty)$ for 
 each $n\in\Z_{\geq0}$ and $\sum_{n=0}^\infty nb_n < \infty$.  
 We will say that such a sequence is \emph{linear summable}. 
 Now we define \[\dtsz: \rxyz \times \rxyz \to \R\] to be given
 by 
% \begin{align*}
% \dtsz\left(\ts, \tsp\right) =& \sum_{i,j\geq 0,(i,j)\neq(0,1)}^\infty \min\{\abs{\si_{i,j}-\si'_{i,j}},b_{i+j}\}\\
%                                            &+ \min\left\{\abs{\si_{0,1} - \si'_{0,1}},2\pi - \abs{\si_{0,1} - \si'_{0,1}}, b_1\right\}.
% \end{align*}
 \[
  \dtsz\left(s, s'\right) = \,\min\Big\{\abs{\si_{0,1} - \si'_{0,1}},2\pi - \abs{\si_{0,1} - \si'_{0,1}}, b_1\Big\} + \sum_{\mathclap{\substack{i,j\geq 0\\(i,j)\neq(0,1)}}}^\infty\, \min\Big\{\abs{\si_{i,j}-\si'_{i,j}},b_{i+j}\Big\}
 \]
 where $s = \ts$ and $s' = \tsp$.
\end{definition}
% \begin{remark}\label{rmk_tsweight}
 This metric is designed to induce the topology in which a sequence of Taylor
 series converges if and only if each term converges.
 Also, notice that two series which agree up to a high order will
 be very close in the metric space and two series which agree only on the 
 high order terms will be distant, as one would expect.
 In Section~\ref{sec_metricts} we develop a similar metric on $\rxy$, which 
 could be of independent interest.
%\end{remark}
\begin{prop}\label{prop_rxyz}
 For any choice of linear summable sequence $\bn$ the space
 $(\rxyz, \dtsz)$ is a complete path-connected metric space 
 and a sequence of Taylor series converges if and only if the
 coefficient of $Y$ converges in $ \R/2\pi\Z$ and all other 
 terms converge in $\R$.  Thus, the topology of $(\rxyz, \dtsz)$
 does not depend on the choice of $\bn$
 as long as it is linear summable.
\end{prop}
Proposition~\ref{prop_rxyz} follows from the proof of 
Proposition~\ref{prop_dts} in Section~\ref{sec_metricts}.

\subsection{Comparing the volume invariant}\label{sec_compareh}
 Since the volume invariant $h_j$ is a real number we simply use 
 the standard metric on $\R$.

\subsection{Comparing the affine invariant}\label{sec_comparelwp}

The topology of spaces of polygons have been studied by many authors. 
For example, in~\cite{HaKn1997, HaKn1998} the authors study polygons 
with a fixed number of edges up to translations and positive homotheties
in Euclidean space and in~\cite{KaMi1995} the authors study polygons 
in $\R^2$ with fixed side length up to orientation preserving isometries.
For this paper we will use a topology on polygons related to the 
Duistermaat-Heckman measure~\cite{DuHe1982} similar to what is done
in~\cite{PePRS2013}.  A natural way to define a metric on closed subsets 
of $\R^2$ is to use the volume of the 
symmetric difference.  Let $\symdiff$ denote 
the symmetric difference of sets.  That is, for $A,B\subset\R^2$ let 
\[
 A\symdiff B = (A \setminus B) \cup (B \setminus A).
\]
In order to define a metric on labeled Delzant semitoric polygons we 
would like to use the volume of the symmetric difference of the polygons
(as is done by Pelayo-Pires-Ratiu-Sabatini in~\cite{PePRS2013}) but there are two problems.  First, 
the polygons here are not required to be compact, so the symmetric 
difference may have infinite volume, and second there are many polygons
to choose from.  To solve the first problem we will use a measure on $\R^2$
which is not the Lebesgue measure.
A natural
choice would be a probability measure on $\R^2$ but the structure of 
$\dpolynok$ is such that vertical translation should not affect the measure.
This is because the elements of $\dpolynok$ are only unique up to specific 
vertical transformations.

\begin{definition}\label{def_adm}
 We say that a measure $\nu$ on $\R^2$ is \emph{admissible} if: 
 \begin{enumerate}
  \item it is in the same measure class as $\mu$, the Lebesgue measure on $\R^2$ (i.e. $\mu\ll\nu$ and $\nu\ll\mu$);
  \item \label{part_admmeastwo} its Radon-Nikodym derivative with respect to Lebesgue measure only depends on the $x$-coordinate, i.e. there exists a $g:\R\to\R$ such that $\nicefrac{\text{d}\nu}{\text{d}\mu}(x,y) = g(x)$ for all $(x,y)\in\R^2$;%maybe say \mu-almost all because rn deriv are only defined ae anyway?
  \item \label{part_admmeasthree} this function $g$ satisfies $g,xg\in\text{L}^1(\mu, \R)$ and on every compact interval $g$ is both bounded and bounded away from zero.
 \end{enumerate}
\end{definition}

\begin{example}
Define $\nu_0$ so that 
\begin{equation*}
\frac{\mathrm{d}\nu_0}{\mathrm{d}\mu} (x,y) = \left\{\begin{array}{ll} 1 &, \text{ if }\abs{x}<1\\ \frac{1}{x^3} &,\text{ else.}\end{array}\right.
\end{equation*}
Notice $x^{-2}\in L^1 (\mu,\R)$ 
and $g_0=\frac{\mathrm{d}\nu_0}{\mathrm{d}\mu} (x,0)$ is bounded and 
bounded away from zero on compact intervals.
Thus, the measure $\nu_0$ is an admissible measure on $\R^2$.
\end{example}
When only considering compact semitoric systems
one can use the Lebesgue measure on $\R^2$ instead to produce a metric 
which induces the same topology, see Remark~\ref{rmk_cpt}.

We say that a map $T:\R^2\to\R^2$ is a \emph{vertical transformation}
if it is of the form $T(x,y) = (x, y + f(x))$ where $f\colon\R\to\R$ is
a measurable function.
Part~\eqref{part_admmeastwo} of Definition~\ref{def_adm} implies that the 
measure is invariant under vertical transformations and part~\eqref{part_admmeasthree}
will force convex sets with finite height at 
every $x$-value to have finite measure. 

\begin{prop}\label{prop_admissible}
 Suppose that $\nu$ is an admissible measure on $\R^2$ and $\De\in\poly$. 
 Then $\De$ has everywhere finite height if and only if $\nu(\De)<\infty$.
\end{prop}

Proposition~\ref{prop_admissible} is proven in Section~\ref{subsec_poly}.

\subsection{Comparing the twisting index}
Notice that the twisting index invariant is more than just a list of integers
defined up to addition by a common constant, it is also the assignment of each
list to the elements of the orbit of a weighted polygon.
Let $\perm$ denote the symmetric group on $\mf$ elements. For $p\in S^\mf$
let the action of $p$ on a vector $\vec{v} = (v_j)_{j=1}^\mf$ by permuting
the elements be denoted by $p(\vec{v}) = (v_{p(j)})_{j=1}^\mf$.

\begin{definition}\label{def_kequiv}
 Suppose $\vec{k},\vec{k}'\in\Z^\mf$ for some nonnegative integer $\mf$
 and let $p\in\perm$.
 Then we say $\vec{k}\sim_p \vec{k}'$ if there exists a constant $c\in\Z$ 
 such that $k_j = k_{p(j)}'+c$ for all 
 $j=1,\ldots,\mf$. We write $\vec{k}\sim\vec{k}'$ if there exists
 a permutation $p\in \perm$ such that $\vec{k}\sim_p \vec{k}'$. 
 We denote by $[\vec{k}]$ the equivalence class of 
 $\vec{k}$ in $\Z^\mf/\sim$.
\end{definition}

\begin{definition}\label{def_tic}
 Let $[\De_w] = [(\De, (\ell_{\la_j}), \ep_j, k_j)_{j=1}^\mf]$
 and $[\De_w'] = [(\De', (\ell_{\la_j'}), \ep_j', k_j')_{j=1}^{\mf'}]$
 be labeled Delzant semitoric polygons.  Then $[\De_w]$ and $[\De_w']$
 are in the same \emph{twisting index class} if $\mf = \mf'$ and $\vec{k}\sim_{\mathrm{Id}}\vec{k}'$.
 They are in the same \emph{generalized twisting index class} if $\mf = \mf'$ and $\vec{k}\sim\vec{k}'$.
 We say that $[(\De, (\ell_{\la_j}), \ep_j, k_j)_{j=1}^\mf]$ is \emph{in the twisting index class}
 $\vec{k}$ and \emph{in the generalized twisting index class $[\vec{k}]$}, using the lists of integers
 to (non-uniquely) label the twisting index classes.
\end{definition}

\begin{remark}
Notice that $[\De_w] = [(\De, (\ell_{\la_j}), \ep_j, k_j)_{j=1}^\mf]$
and $[\De_w'] = [(\De', (\ell_{\la_j'}), \ep_j', k_j')_{j=1}^{\mf'}]$
are in the same twisting index class if and only if $\mf = \mf'$ and the representatives
can be chosen so that $k_j = k_j'$ for all $j=1, \ldots, \mf$.
Also notice that it is possible for two systems to be in the same twisting index class
but have different twisting index invariants.
\end{remark}

\begin{definition}\label{def_permkkp}
 Fix any $\vec{k},\vec{k}'\in\Z^\mf$ such that $\vec{k}\sim\vec{k}'$.
 Let 
 \[
  \permkkp = \left\{p\in\perm \left| \begin{array}{l}\text{there exists }c\in\Z \text{ such that }\\k_j = k_{p(j)}'+c \text{ for all }j=1,\ldots,\mf \end{array}\right. \right\}.
 \]
 Notice that $\vec{k}\sim\vec{k}'$ is equivalent to 
 $\permkkp\neq\varnothing$.  The elements of $\permkkp$ will be called 
 \emph{appropriate permutations} for $\vec{k}$ and $\vec{k}'$.
\end{definition}

If $[\De_w] = [(\De, (\ell_{\la_j}), \ep_j, k_j)_{j=1}^\mf]$
 and $[\De_w'] = [(\De', (\ell_{\la_j'}), \ep_j', k_j')_{j=1}^{\mf'}]$
 are in the same generalized twisting index class then the representatives
 can be chosen so that $k_j = k_{p(j)}'$ for $j=1, \ldots, \mf$.  There are
 only finitely many elements of $[\De_w]$ with any given fixed $\vec{k}$, so
 after fixing these integers to compare two labeled Delzant semitoric polygons
 we can take the symmetric difference of each corresponding pair and sum
 them.  
 According to Remark~\ref{rmk_e}, we can cycle through all possible
 polygons of a fixed $\vec{k}$ by starting with one representative
 for which $\ep_j=+1$ for $j=1, \ldots, \mf$ as in the following definition.

%Now, assume that two labeled weighted polygons have the same number of 
%focus-focus points and twisting indices related by $\sim$. We can shift
%the twisting index of one of the labeled weighted polygons by the action
%of an element of $\tk$ such that after the shift the 
%two labeled weighted polygons in question will have the same twisting 
%index modulo the ordering. Once the twisting indices are fixed we still
%have a family of polygons which depends on the choice of 
%$\vec{\ep}\in\{-1,+1\}^\mf$.  The number of possible choices of $\vec{\ep}$
%is finite so we will simply sum up the symmetric difference of each pair of
%polygons for each choice of $\vec{\ep}$. Using Remark~\ref{rmk_e} we can 
%concisely write this in the following way.

\begin{definition}\label{def_distpolyp}
 Suppose that for $i=1, 2$ we have 
 $[\De_w]_i = [(\De^i, (\ell_{\la_j}^i, +1, k^i_j)_{j=1}^\mf)]\in\dpolynok$
 for some $\mf>0$ and with $\vec{k}\sim\vec{k}'$, so $\permkkp\neq\varnothing$.
 For $p\in\permkkp$ define 
 \[
  \distpolyp\big([\De_w]_1,[\De_w]_2\big) = \,\sum_{\mathclap{\vec{u}\in\{0,1\}^\mf}}\, \nu\big( t_{\vec{u}}(\De^1) \symdiff t_{p(\vec{u})}(T^{-c}(\De^2)) \big)
 \]
 where $c\in\Z$ is the unique integer such that $k_j - k_{p(j)}' = c$ 
 for all $j=1,\ldots,\mf$.  In the case that $[(\De)],[(\De')]\in\dpolynokz$
 define 
 \[
  \distpoly([(\De)],[(\De')])= \nu( \De \symdiff \De' ).
 \]
\end{definition}

If $\mf=0$ the labeled weighted polygon becomes a single polygon.
The definition of $\distpolyp$ in this case
should be thought of as the same formula as the $\mf>0$ case and it is only
treated separately because the sum in the more general formula would be 
empty if $\mf=0$.

Notice that $\distpolyp$ is not a metric if $p\neq p^{-1}$ in $\perm$ 
because it is not symmetric. 
We will remove the dependence on a choice of permutation in the next section 
when we define the final version of the metric.  There are many 
ways to choose a representative from each equivalence class which have matching 
twisting indices,
but the volume of the symmetric difference will not actually depend on that choice
(see Proposition \ref{prop_tkvolume}) so this function is well-defined on orbits 
of $\Gmf\times\tk$.

\subsection{Definition of metric and main result}\label{sec_gencase}

For the metric we present in this paper we automatically force
systems which are in different generalized twisting index classes
(see Definition~\ref{def_tic}) are in different components
by declaring the distance between them to be infinite.
In particular, this implies that systems with a different number
of focus-focus singular points are in different components of $\semitoric$.
In fact, we will see that system which are in the same generalized twisting
index class but not in the same twisting index class are in different components,
but the distance between such systems is not defined to be infinite
(see Remark \ref{rmk_components}).

\begin{definition}\label{def_semitoricmk}
 Suppose that $\mf\in\Z_{>0}$ and $\vec{k}\in\Z^\mf$.  Then we define 
 $\semitoricmk\subset\semitoricm$ to be those elements with some representative 
 of their Delzant semitoric polygon invariant having integers labeling $\vec{k}$ and define 
 \[
  \semitoricmkperm = \bigcup_{\vec{k}'\in[\vec{k}]} \semitoricmkprime.
 \]
 This means that $\semitoricmk$ is a twisting index class and $\semitoricmkperm$
 is a generalized twisting index class (as in Definition~\ref{def_tic}).
 Furthermore, define 
 \[
  \dpoly = \{\lwpkp\in\dpolynok \mid \vec{k}\sim\vec{k}'\}
 \]
 and 
 \[
  \mmk = \mm \cap \left( \dpoly\times\ \R^\mf\times\rxyz^\mf \right).
 \]
\end{definition}

%\begin{remark}
Notice that 
\[
 \semitoric = \bigcup_{\substack{\mf\in\Z_{\geq0}\\ \vec{k}\in\Z^\mf}} \semitoricmk.
\]
This union, and the union in Definition \ref{def_semitoricmk}, are not 
disjoint unions only because they have repeated terms.  For instance, 
since the action of $\tk$ can shift all of the twisting indices, we 
have that 
\[
 \semitoricmk = \semitoricmkc
\]
for any $c\in\Z$.
%\end{remark}

From Sections \ref{sec_comparets}, \ref{sec_compareh}, and \ref{sec_comparelwp},
given some fixed appropriate permutation we already know how to define 
a ``distance'' function  on two systems with specified twisting index class. 
To produce a metric which does not depend on fixing a permutation we will
take the minimum of each possibility.

\begin{definition}\label{def_metric}
 Let $\mf\in\Z_{\geq0}$ and $\vec{k}\in\Z^\mf$ and suppose that 
 $\mele,\melep\in\mmk$ with 
 \[
  \mele=\stfull\textrm{ and }\melep=\stfullp.
  \]
 Let $\nu$ be an admissible measure, $\bn$ be a linear summable 
 sequence, and $p\in\permkkp$. We define:
 \begin{enumerate}
  \item\label{item_metricperm} the \emph{comparison with alignment} $p$ to be 
        \[
         \Dp (\mele,\melep)= \distpolyp ([\dew],[\dewp]) + \sum_{j=1}^\mf \Big( \dtsz((S_j)^\infty, (S_{p(j)}')^\infty) + \big|h_j - h_{p(j)}'\big|\Big);
        \]
  \item the \emph{distance between $\mele$ and $\melep$} to be 
        \[
         \D (\mele, \melep) = \min_{p\in \permkkp} \left\{\Dp(\mele,\melep) \right\}. 
        \]
 \end{enumerate}
\end{definition}
  
 A minimum of even a finite number of metrics is not a metric in general,
 but we will see in Theorem~\ref{thm_main} that $\D$ is a metric in this case.  
 Now we use this distance defined on each $\mmk$ to induce a distance
 on the whole space which can be pulled back to produce a metric on $\semitoric$.

\begin{definition}\label{def_fullmetric}
 Let $\nu$ be an admissible measure and $\bn$ be a linear 
 summable sequence.  Then we define
 \begin{enumerate}
  \item the \emph{distance on $\m$} by 
        \[
         \Dnomk(\mele,\melep)=\left\{\begin{array}{ll}\D(\mele,\melep)&,\text{ if }\mele,\melep\in\mmk\text{ for some }\mf\in\Z,\vec{k}\in\Z^\mf\\\infty&,\text{ otherwise}\end{array}\right. 
        \]
        for $\mele,\melep\in\m$;
  \item the \emph{distance on} $\semitoric$ by $\Dst = \Phi^* \Dnomk$ 
        where $\Phi: \semitoric \to \m$ is the bijective correspondence 
        from Theorem \ref{thm_class}.
 \end{enumerate}
\end{definition}

Notice that the distance between two systems is finite if and only if
they are in the same generalized twisting index class (Definition~\ref{def_tic}).
To state the main theorem we will have to first define the completion.

\begin{definition}\label{def_mcomplmk}
 For any choice of $\mf\in\Z_{\geq0}$ and $\vec{k}\in\Z^\mf$ we define 
 \[
  \mcomplmk = \polycompl\times[0,1]^\mf\times\rxyz^\mf
 \]
 and 
 \[
  \mcompl = \bigcup_{\substack{\mf\in\Z_{\geq0}\\\vec{k}\in\Z^\mf}} \mcomplmk
 \]
 where the critical points satisfy the ordering convention from 
 Remark~\ref{rmk_order} and $\polycompl$ is defined as in Definition~\ref{def_polycompl}.
\end{definition}

 \begin{lettertheorem}\label{thm_main}
  For any choice of
  \begin{enumerate}[font=\normalfont]
   \item a linear summable sequence $\bn$;
   \item an admissible measure $\nu$;
  \end{enumerate}
the space $\metricspace$ is a non-complete metric space whose completion 
corresponds to $\mcompl$.  Moreover, the topology of $\metricspace$ is 
independent of the choice of $\nu$ and $\bn$.
 \end{lettertheorem}

\begin{remark}
  There are several important facts to notice about Theorem~\ref{thm_main}:
 \begin{enumerate}

  \item The distance $\Dst$ depends on the choice of $\bn$ and $\nu$, but the induced topology
        does not so the family of distances introduced in Definition~\ref{def_fullmetric} 
        induces a unique topology on $\semitoric$ and thus 
        Theorem~\ref{thm_main} completely resolves Problem 2.43 
        asked by Pelayo-V\~{u} Ng\d{o}c in~\cite{PeVNfirst2012}.
        
  \item If $\mele,\melep\in\mmk\text{ for some }\mf\in\Z,\vec{k}\in\Z^\mf$ then $\Dnomk(\mele,\melep)<\infty$. 
  
  \item In special cases a less complicated form of the metric can be used.
        Let $\mathrm{Id}\in\mathcal{S}^{\mf}$ denote the identity permutation
        so $\Did$ is the comparison with alignment $p$
        from Definition~\ref{def_metric} part~\eqref{item_metricperm} in the case
        that $p=\mathrm{Id}$.
        The metric 
        \[
         \mathcal{D}^{\mathrm{Id}, \nu, \bn}:=\Phi^*\Did
        \] 
        is easier to work with and induces the same
        topology as $\Dst$ (Proposition~\ref{prop_sametop}) so this should
        be used to study topological properties of $\semitoric$.
        Additionally, when studying compact semitoric systems the admissible 
        measure on $\R^2$ can be instead replaced by the standard Lebesgue 
        measure without changing the topology (Remark~\ref{rmk_cpt}).
        See Example~\ref{ex_needmin} for an explanation of why $\Dst$ 
        produces the appropriate metric space structure on $\semitoric$.

  \item Since toric integrable systems fall into the broader category of 
        semitoric systems it is natural to wonder if the metric defined 
        in this paper is compatible with the metric on toric systems 
        by Pelayo-Pires-Ratiu-Sabatini in~\cite{PePRS2013}.  Because we must choose an admissible measure 
        to apply to the more general cases the metric induced by $\Dplain$ 
        does not exactly match the metric defined on toric systems but 
        they do induce the same topology, see Section \ref{sec_toric}.
        
  \item Since all metric spaces are Tychonoff (completely regular and 
        Hausdorff) we know that $\semitoric$ is Tychonoff.  Thus the 
        Stone-C\v{e}ch compactification~\cite{Ce1937,St1937} applies 
        to $\semitoric$ so it admits a Hausdorff compactification 
        (just as in~\cite{PePRS2013}).

%   \item Since an integrable system is a manifold and a map into 
%         $\R^n$ for some $n\in\N$ one may consider using a general metric 
%         on maps to define a metric on integrable systems. A metric defined 
%         on collections of maps with varying domains as is 
%         in~\cite{Pa2015}, while very general, would actually not be appropriate 
%         in this situation because the singularities, and thus isomorphism 
%         type, of semitoric systems can be changed by perturbing the 
%         systems on arbitrarily small sets.  For instance, this can be seen 
%         because the Taylor series invariant is completely independent of 
%         the other invariants.  For this reason we have defined a metric on 
%         the invariants, which describe the essential properties of the 
%         integrable system, to produce an appropriate metric on $\semitoric$.

 \end{enumerate}
\end{remark}

\section{The metric}\label{sec_themetric}
 In this section we fill in the details of constructing the metric and prove that it is a metric. 
\subsection{Metrics on Taylor series}\label{sec_metricts}
Let $\rxy$ refer to the algebra of real formal power series in two variables, $X$ and $Y$.
\begin{definition}
 Suppose that $\bn$ is any linear summable sequence.  Then we define the 
 \emph{distance on Taylor series} to be the function 
 \[
  \dts : \rxy \times \rxy \to \R 
 \] 
 given by 
 \[
  \dts\left(\ts, \tsp\right) = \sum_{i,j=0}^\infty \min\left\{\abs{\si_{i,j}-\si'_{i,j}},b_{i+j}\right\}.
 \]
\end{definition}
\begin{prop}\label{prop_dts}
 The space $(\rxy, \dts)$ is a complete path-connected metric space and 
 a sequence of Taylor series $\left( s_k = \sum_{i,j\geq0}\sigma_{i,j}^k X^i Y^j\right)_k$
 converges if and only if each sequence of terms $(\sigma_{i,j}^k)_k$ converges.
\end{prop}
\begin{proof}
 First notice that the sum in the definition of the distance always converges.  
 This is because 
 \[
  \dts\left(\ts, \tsp\right) \leq \sum_{i,j=0}^\infty b_{i+j} = \sum_{n=0}^\infty (n+1)b_n< \infty
 \]
 for any pair of Taylor series by the choice of $\bn$.
 It is also clear that $\dts$ is symmetric and positive definite.
 It satisfies the triangle inequality because that inequality is satisfied 
 for each term and thus we can see that $(\rxy, \dts)$ is a metric space.
 
 Next we will prove the condition on convergence.  
 Suppose that 
 \[
  \lim_{k\to\infty} \dts \left( s_k, s_0 \right) = 0
 \]
 with $s_k, s_0 \in \rxy$ as in the statement of the Proposition.
 Fix any $I,J\in\Z_{\geq0}$
 and we will show that $\si_{I,J}^k \stackrel{k\to\infty}{\longrightarrow} \si_{I,J}^0$.
 Fix $\varepsilon>0$ and find $K$ such that $k>K$ implies that 
 \[
  \sum_{i,j=0}^\infty \min\left\{\abs{\si_{i,j}^k-\si_{i,j}^0},b_{i+j}\right\}<\varepsilon
 \]
 because we may assume that $\varepsilon<b_{I+J}$.
 Then we can see that $\abs{\si_{I,J}^k - \si_{I,J}^0}<\varepsilon$ so the result follows.
 
 Now we will show the converse.  Suppose that 
 \[
  \lim_{k\to\infty} \abs{\si_{i,j}^k - \si_{i,j}^0}=0
 \]
 for all $i,j\in\Z_{\geq0}$. Fix $\varepsilon>0$, let $N\in\N$ be such that 
 \[
  \sum_{n\geq N} (n+1)b_n < \frac{\varepsilon}{2},
 \]
 and let $K\in\Z$ be such that $k>K$ implies that 
 \[
  \abs{\si_{i,j}^k - \si_{i,j}^0}<\frac{\varepsilon}{N(N+1)}
 \]
 for each $i,j\in\Z_{\geq0}$ such that $i+j<N$.  Notice it is possible to 
 do this simultaneously because there are only finitely many such pairs $(i,j)$.
 For any $k>K$ we have that
 \begin{align*}
  \dts \left( s_k, s_0\right) & \leq \sum_{i+j<N}\abs{\si_{i,j}^k - \si_{i,j}^0} + \sum_{i+j\geq N} b_{i+j}\\
                                  & < \frac{\varepsilon}{N(N+1)} \sum_{i+k<N} 1 + \sum_{n\geq N} (n+1)b_n\\
                                  & < \frac{\varepsilon}{N(N+1)} \frac{N(N+1)}{2} + \frac{\varepsilon}{2} = \varepsilon.
 \end{align*}
 This proves the convergence condition.
 
 Any element of this space may be continuously transformed into 
 any other linearly in each term, so it is path-connected. To finish
 the proof we will show that this space is complete.  Suppose that 
 $\left( s_k \right)_{k=0}^\infty$ is a Cauchy sequence in $\rxy$.
 Using an argument similar to the one for convergence, we can see that
 the sequence $\{\si_{i,j}^k\}_{k=0}^\infty$ is Cauchy for each
 $i,j\in\Z_{\geq0}$ and therefore $\si_{i,j}^k \stackrel{k\to\infty}{\longrightarrow} \si_{i,j}^0$
 for some $\si_{i,j}^0\in\R$. Since it converges
 in each term, we can use the convergence condition to conclude that 
 \[
  \lim_{k\to\infty} \dts \left( s_k, \tsz \right) = 0
 \]
 and so all Cauchy sequences have limits.
\end{proof}

We have characterized convergence in this space in a way which is independent of the sequence $\bn$.
Since the topology of a metrizable space is completely determined by its convergent sequences 
we have the following result.

\begin{cor}
 The topology on $\rxy$ determined by $\dts$ does not depend on the choice of the linear summable sequence $\bn$.
\end{cor}

Notice that $\rxyz$ is not a closed subset of $(\rxy, \dts)$ and $(\rxyz, \dts)$
with the restricted metric is not a complete metric space.  To see this
consider any collection of Taylor series in which $\si_2 \to 2\pi$.  
This does not accurately describe the structure of the semitoric 
systems ($\si_2\in[0,2\pi)$ represents a point on the circle, see the construction of the Taylor series invariant~\cite{VN2003}) and 
thus we use the altered metric from Definition~\ref{def_metricts}.
Proposition~\ref{prop_rxyz} follows from a slightly altered version of 
the proof of Proposition~\ref{prop_dts}.

\begin{remark}\label{rmk_tsmorevar}
 A similar construction to $\dts$ can be used to produce such a metric on Taylor 
 series in any number of variables.  The only difference is that to produce a 
 metric on Taylor series in $m$ variables the sequence $\bn$ would be 
 required to satisfy 
 \[
  \sum_{n=0}^\infty\binom{n+m-1}{n}b_n<\infty
 \]
 because there are $\binom{n+m-1}{n}$ terms of degree $n$ in a Taylor series 
 on $m$ variables.
\end{remark}

\subsection{Metrics on labeled weighted polygons}\label{subsec_poly}

We start this section with a proof.
\begin{proof}[Proof of Proposition \ref{prop_admissible}]
 First suppose that $\De\in\poly$ has everywhere finite height and we
 will show that $\nu(\De)$ is finite.
 By definition $\De$ is the intersection of half-spaces and since it is 
 assumed to have everywhere finite height we can see that this collection 
 of half spaces must include at least two which are not completely vertical,
 i.e.~not of the form $\{x\geq c\}$ or $\{x\leq c\}$ for $c\in\R$.  
 Let $B$ denote the intersection of these two half planes.  Then by definition 
 $\De\subset B$ and thus $\nu(\De)<\nu(B)$.  If the two half planes are parallel
 of a distance $c$ apart then 
 \[
  \nu(B) = \varint_{B}\dnudmu\dmu = \varint_{\R} cg \dmu<\infty
 \] 
 because $g\in L^1(\mu,\R)$.  If the spaces are 
 not parallel then their boundaries intersect at some point $(x_0,y_0)$.  
 Let $m$ be the absolute value of the difference in the slopes of the two boundaries.
 Then for each value $(x,y)\in\R^2$ the height of $B$ at that $x$-coordinate is 
 $m\abs{x-x_0}$ and the sign of $x-x_0$ is the same for each $(x,y)\in B$.  
 Assume that $x-x_0 \geq 0$ for all $(x,y)\in B$ so we have 
 \[
  \nu(B) = \varint_B \dnudmu\dmu = \varint_{x_0}^\infty m(x-x_0)g(x)\dmu = m \varint_{x_0}^\infty xg\dmu - m x_0 \varint_{x_0}^\infty g\dmu <\infty
 \]
 because $g\in L^1(\mu, \R)$ and $xg\in L^1 (\mu,\R)$.  The computation is similar if
 $x-x_0\leq 0$ for all $(x,y)\in B$.

Next we show any compact polygon $\De\in\poly$ without everywhere finite height will have infinite $\nu$-measure.  
This is because a compact polygon which does not have everywhere finite height 
includes a subset of 
the form $\{(x,y)\mid a_1<x<a_2, \varepsilon y>b\}$ for some $a_1<a_2$, $b\in\R$, $\varepsilon\in\{\pm 1\}$
(otherwise it is a vertical line, which is not a polygon).  Such a subset has 
infinite $\nu$-measure because $\nu$ is invariant under vertical translations.
\end{proof}

Even once we have fixed the cut directions there are many polygons to choose 
from based on the choice of the twisting index (i.e.~the orbit of the action 
of $\tk$) but if the same choice is made for each pair of polygons 
this choice does not change the volume of the symmetric difference.

\begin{prop}\label{prop_tkvolume}
 Let $\mf\in\Z_{\geq0}$, $p\in\perm$, and consider
 \[
  \mathcal{J}^p = \left\{ ([\lwp],[\lwpp])\in (\dpoly)^2 \middle| p\in\permkkp\right\}.
 \]
 Then the function $\distpolyp:\mathcal{J}^p\to\R$ is well defined.
\end{prop}

\begin{proof}
 Suppose that 
 \[
  \De_w^1=\lwpplusone, \qquad \De_w^2 = \lwpplustwo\in[\lwp]
 \] 
 and $\De'_w = \big(\De', (\ell_{\la_j'}, +1, k_j')_{j=1}^\mf\big)$.  Then there exists some $d\in\Z$ such 
 that $k^1_j=k^2_j-d$ for $j=1,\ldots,\mf$ and $\De^1 = T^d(\De^2)$. 
 Since $p\in\permkkp$, there exists $c\in\Z$ such that $k_j^1 - k_{p(j)}' = c$
 for all $j$ and notice that this 
 means that $k_j^2 - k_{p(j)}' = c + d$ for all $j$. Therefore,
 \begin{align*}
  \distpolyp([\De_w^1],[\De'_w] ) &= \,\,\,\,\sum_{\mathclap{\vec{u}\in\{0,1\}^\mf}}\,\nu\big( t_{\vec{u}}(\De^1) \symdiff t_{p(\vec{u})}(T^{-c}(\De')) \big)\\
  &= \,\,\,\,\sum_{\mathclap{\vec{u}\in\{0,1\}^\mf}}\,\nu\big( t_{\vec{u}}(T^d(\De^2)) \symdiff t_{p(\vec{u})}(T^{-c}(\De')) \big)\\
  &= \,\,\,\,\sum_{\mathclap{\vec{u}\in\{0,1\}^\mf}}\,\nu\big(T^d (t_{\vec{u}}(\De^2) \symdiff t_{p(\vec{u})}(T^{-c-d}(\De'))) \big)\\
  &= \,\,\,\,\sum_{\mathclap{\vec{u}\in\{0,1\}^\mf}}\,\nu\big(t_{\vec{u}}(\De^2) \symdiff t_{p(\vec{u})}(T^{-(c+d)}(\De')) \big)\\
  &= \distpolyp([\De_w^2], [\De'_w])
 \end{align*}
 because admissible measures are invariant under vertical transformations 
 such as $T^d$.  The argument that this function is well defined in the 
 second input is similar.
\end{proof}

\subsection{Choice of \texorpdfstring{$\nu$}{the admissible measure} does not change the topology}

While the choice of admissible measure will change the metric it does 
not change the topology induced by that metric.  

\begin{lemma}\label{lem_convexinterval}
 Suppose that $\nu$ is an admissible measure and
 $\De_k, \De\in\poly$ for $k\in\N$ are such that 
 $\nu(\De_k\symdiff\De)\stackrel{k\to\infty}{\longrightarrow}0$.  
 Then there exists a vertical segment 
 $A=\{x_0\}\times[y_0,y_1]$, with $y_0<y_1$, and $K>0$ such
 that $A\subset \De_k \cap \De$ for all $k>K$.
\end{lemma}

\begin{proof}
 Fix any $N>0$ such that $\De\cap([-N,N]\times\R)$ has non-zero measure 
 with respect to $\nu$, and thus also with respect to $\mu$.  Since 
 $\nu$ is admissible we can find some $c>0$ such that 
 $\nicefrac{\text{d}\nu}{\text{d}\mu} > c$ on $[-N,N]\times\R$.
 
 For each $\varepsilon>0$ let 
 \[
  U_\varepsilon = \left\{ p\in\R^2 \middle| B_{\varepsilon}(p)\subset (-N,N)\times\R\big) \cap \De\right\}
 \] 
 where $B_{\varepsilon}(p)$ is the standard open ball of radius $\varepsilon$ centered at $p$
 and $\mathrm{int}(A)$ denotes the interior of the set $A$.

 Fix any $k\in\N$ and suppose $U_\varepsilon \setminus \De_k\neq \varnothing$.
 Because $\De_k$ is the intersection of closed half-planes 
 its complement, $\De_k^c$, is the union of open half-planes.  
 If $q\in U_\varepsilon \setminus \De_k$ then there exists some open 
 half-plane with boundary including $q$ which is a subset of $\De_k^c$.  Let $H_q$ 
 be the intersection of one such half-plane with $B_{\varepsilon}(q)$ so $H_q\subset \De\setminus \De_k$.
 Then, since $H_q\subset (-N,N)\times\R$,
 \[
  \nu(H_p)=\varint_{H_p}\frac{\text{d}\nu}{\text{d}\mu}\text{d}\mu > c \mu(H_p) = \frac{c}{2}\mu\big(B_{\varepsilon}(p)\big) = \frac{c \pi}{2} \varepsilon^2.
 \]
 Thus, 
 if $U_\varepsilon \setminus \De_k$ is non-empty then 
 $\nu(\De_k\symdiff\De)>\frac{c \pi}{2} \varepsilon^2$.
 
 Now choose $\varepsilon$ small enough that $U_\varepsilon$ is non-empty and 
 choose $K>0$ such that $k>K$ implies that $\nu(\De\symdiff\De_k)<\frac{c \pi}{2}\varepsilon^2$.
 If $U_\varepsilon \setminus \De_k\neq \varnothing$
 then $\nu(\De_k\symdiff\De)>\frac{c \pi}{2} \varepsilon^2$, so
 we conclude $U_\varepsilon \subset \De_k$ for $k>K$.  
 The set $U_\varepsilon$ has nonempty interior so we can find the 
 set $A$ as in the statement of the Lemma.
\end{proof}

Now we will use Lemma~\ref{lem_convexinterval} to prove Lemma~\ref{lem_admmsrseq},
which says that the same sequences of polygons converge with respect to any admissible measure.

\begin{lemma}\label{lem_admmsrseq}
 Suppose that $\nu_1,\nu_2$ are admissible measures and that $\De_k, \De\in\poly$ 
 for $k\in\N$ have $\nu_1(\De),\nu_1(\De_k)<\infty$. If 
 $\nu_1(\De_k\symdiff\De)\stackrel{k\to\infty}{\longrightarrow}0$
 then $\nu_2(\De_k\symdiff\De)\stackrel{k\to\infty}{\longrightarrow}0$.
\end{lemma}

\begin{proof}
 Suppose that $\nu_1(\De_k\symdiff\De)\stackrel{k\to\infty}{\longrightarrow}0$ and let 
 $A,x_0,y_0$, and $y_1$ be as in Lemma~\ref{lem_convexinterval}.  We know that 
 the line $\{x=x_0\}$ intersects $\De$ so it must intersect the top boundary of
 $\De$, since $\De$ has everywhere finite height by Proposition~\ref{prop_admissible}.
 Since a convex set is the intersection of half-planes there must exist a line $\ell_1$
 which goes through the point where $\{x=x_0\}$ intersects the top boundary such that 
 all of $\De$ is in a closed half-plane bounded by $\ell_1$ (as in Figure \ref{fig_admlem1}).
 Such a line may not be unique if there is a vertex with $x$-coordinate equal to $x_0$, 
 but any choice of such a line will do.
  
\begin{figure}
 \centering
 \includegraphics[height=80pt]{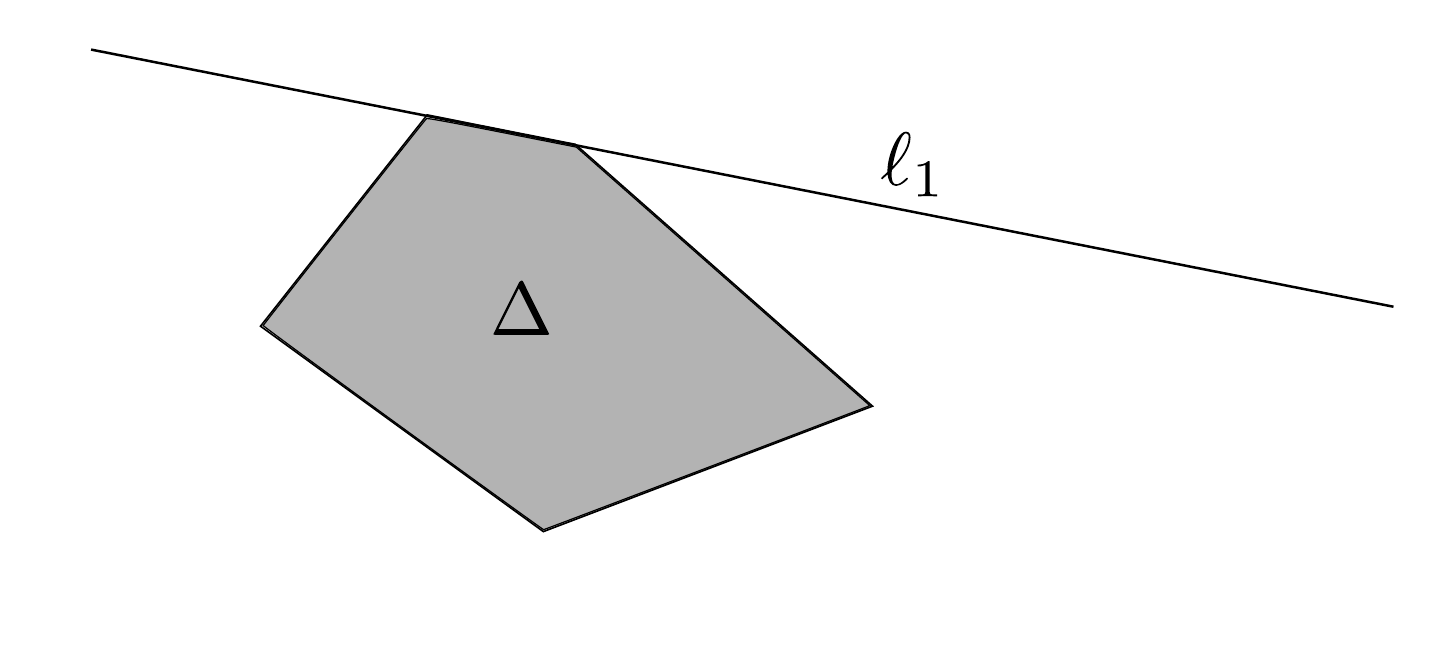}
 \caption{Since $\De$ is convex it must all lie on the same side of $\ell_1$.}
 \label{fig_admlem1}
\end{figure}

The situation we describe next is shown in Figure~\ref{fig_admlem2}. 
Let $m$ denote the slope of $\ell_1$ and let $\ell_2$ be the line 
through $(x_0,y_1)$ with slope $m+1$. Let $m'$ denote the slope of 
the line through the point $(x_0,y_0)$ and the point which is the 
intersection of $\ell_1$ with $\ell_2$. Finally let $\ell_3$ be the
line through $(x_0,y_0)$ with slope $\nicefrac{(m+m')}{2}$.  Since 
the slope of $\ell_3$ is greater than the slope of $\ell_2$ these 
two lines must intersect at some $x$-coordinate greater than $x_0$,
but since the slope of $\ell_3$ is less than $m'$ we know that the 
intersection of $\ell_2$ and $\ell_3$ must be to the right of
the intersection of $\ell_1$ and $\ell_2$.  Thus the lines 
$\ell_1, \ell_2$, and $\ell_3$ bound a triangle which we will denote
by $G$, as is shown in Figure \ref{fig_admlem2}.  
Let $N_1 = \max_{s\in G}\pi_1 (s)$.  Since $\De$ is on one side
of $\ell_1$ and $G$ is on the other we conclude that $G\cap \De = \varnothing$.

\begin{figure}[b]
 \centering
 \includegraphics[height=150pt]{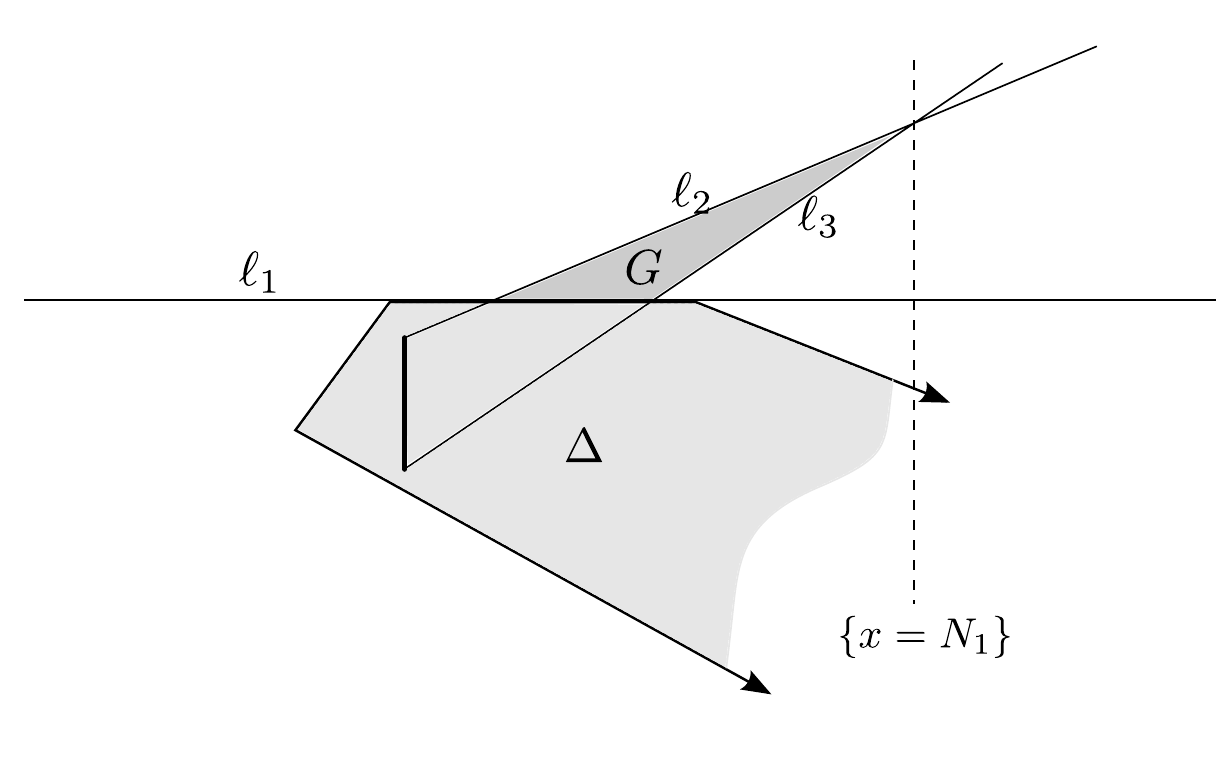}
 \caption{The lines $\ell_1,\ell_2,\ell_3$, and triangle $G$.}
 \label{fig_admlem2}
\end{figure}

For any $N>x_0$ let $E_N^1$ denote the region of $\R^2$ 
which has $x>N$ and is above or on $\ell_2$.
Now suppose that $k$ is large enough so that $A\subset\De_k$ and 
let $p\in E_{N_1}^1\cap\ell_2$.
Then $p\in\De_k$ implies that $G\subset\De_k\symdiff\De$ because $\De_k$ 
is convex and $\De\cap G=\varnothing$.  Similarly, if $p$ is any other 
point in $E^1_{N_1}$ we can conclude that some $\nu_1$-preserving 
transformation of $G$ must be contained in $\De_k\symdiff\De$.  This is
because moving $p$ vertically will result in acting on $G$ by some 
matrix $T^r$ (as in Equation \eqref{eqn_tk}) with $r\in\R$ with origin
on the line $\{x=x_0\}$ (see Figure \ref{fig_admlem4}).  In any case, 
if $\De_k\cap E_{N_1}^1$ is nonempty and $k$ is large enough so that 
$A\subset\De_k$ then we can conclude that 
$\nu_1 (\De\symdiff\De_k)\geq\nu_1 (G)>0$.  
Since $\nu_1 (\De\symdiff\De_k)\stackrel{k\to\infty}{\longrightarrow} 0$ 
we can conclude that for large enough $k$ the set $\De_k\cap E_{N}^1$ is empty.

\begin{figure}
 \centering
 \includegraphics[height=80pt]{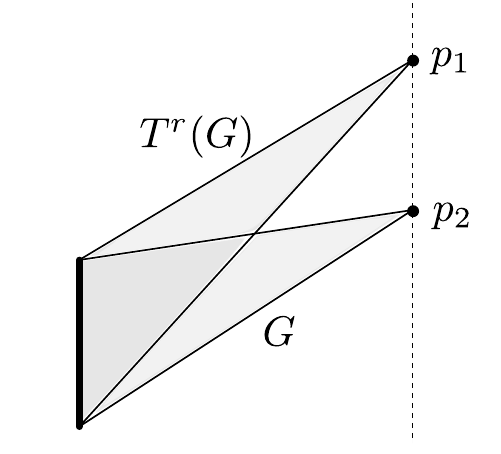}
 \caption{Notice that for a fixed vertical line segment $A\subset\R^2$ the measure 
          of the convex hull of $A$ and $p\in\R^2$ only depends on the $x$-component
          of $p$.  This is because if $p_1,p_2\in\R^2$ with $\pi_1(p_1)=\pi_2(p_2)$
          then the convex hulls are related by a vertical transformation.}
 \label{fig_admlem4}
\end{figure}

Using a similar argument, one can define sets $E_{N}^i$ for $i=2,3,4$ that 
must also be disjoint from $\De_k$ for large enough $k$ and $N$;
these are shown in Figure~\ref{fig_admlem3}.
The sets $E_{N}^1$ and $E_{N}^2$ are bounded to the left by the line $\{x=N\}$
and the sets $E_{N}^3$ and $E_{N}^4$ are bounded to the right by $\{x=-N\}$. 
The sets $E_{N}^1$ and $E_{N}^4$ are bounded below by lines and the sets 
$E_{N}^2$ and $E_{N}^3$ are bounded above by lines.  Let 
$E_N = \cup_{i=1}^4 E_{N}^i$ and let $N_2>N_1$ be large enough so that for
large enough $k$ we have that $\De_k \cap E_{N_2}=\varnothing$.  Let 
$D_N=[-N,N]\times\R$ for $N\in\R$ and let $S_N = \R^2 \setminus (E_N \cup D_N)$.

Fix $\varepsilon>0$.  Notice that for each $N>0$ the set $S_N$ is of finite $\nu_2$-measure.  
Since $\{S_N\}_{N>0}$ are nested we conclude that $\lim_{N\to\infty} \nu_2 (S_N) = 0$.  
Now choose some fixed $N_3>N_2$ and $K_1>0$ such that $\nu_2(S_{N_3})<\varepsilon$ 
and $k>K_1$ implies that $\De_k \cap E_{N_3}=\varnothing$.  
Since both $\nu_1$ and $\nu_2$ are admissible measures we know that their 
Radon-Nikodym derivative is bounded on $D_{N_3}$.  This is because 
\[
 \frac{\mathrm{d}\nu_2}{\mathrm{d}\nu_1}=\frac{\mathrm{d}\nu_2}{\mathrm{d}\mu}\left(\frac{\mathrm{d}\nu_1}{\mathrm{d}\mu}\right)^{-1},
\]
which are both bounded on $D_{N_3}$.  Let $c>0$ be such that 
$\nicefrac{\text{d}\nu_2}{\text{d}\nu_1}<c$ on $D_{N_3}$.  
Now choose $K_2>K_1$ such that $k>K_2$ implies 
$\nu_1(\De\symdiff\De_k)<\varepsilon$.
Finally, for $k>K_2$ we have
\begin{align*}
 \nu_2(\De_k\symdiff\De)&=\varint_{\R^2} \abs{\chi_{\De_k} - \chi_\De}\, \mathrm{d}\nu_2\\
                        &=\varint_{S_{N_3}} \abs{\chi_{\De_k} - \chi_\De}\, \mathrm{d}\nu_2 + \varint_{E_{N_3}} \abs{\chi_{\De_k} - \chi_\De}\, \mathrm{d}\nu_2 + \varint_{D_{N_3}} \abs{\chi_{\De_k} - \chi_\De}\, \mathrm{d}\nu_2\\
                        &\leq \nu_2(S_{N_3}) + 0 + \varint_{D_{N_3}} \abs{\chi_{\De_k} - \chi_\De}\frac{\text{d}\nu_2}{\text{d}\nu_1}\, \mathrm{d}\nu_1\\
                        &< \varepsilon + c \, \nu_1(\De_k\symdiff\De)\\
                        &< (1+c) \varepsilon,
\end{align*}
which can be made arbitrarily small.
\end{proof}

\begin{figure}[t]
 \centering
 \includegraphics[height=120pt]{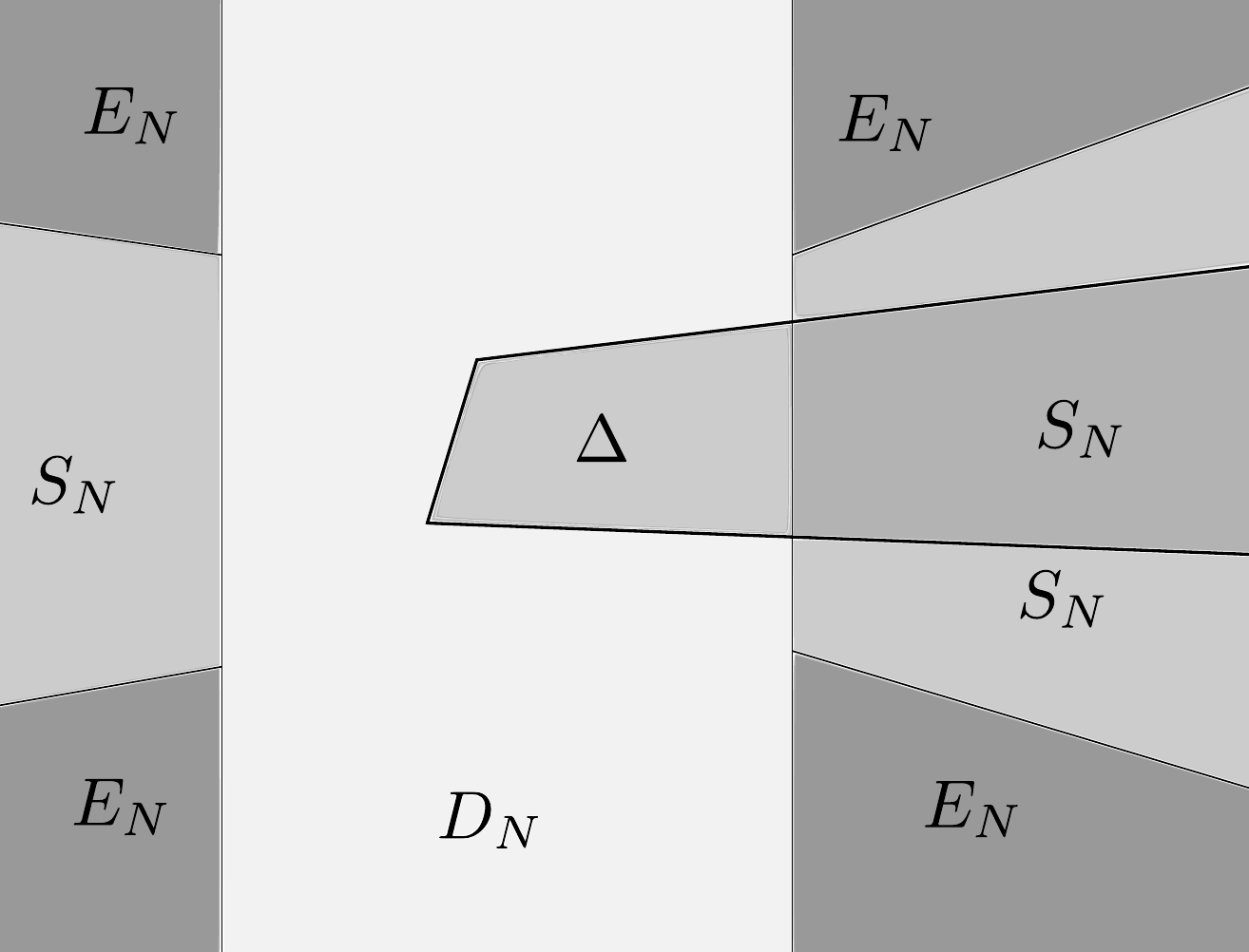}
 \caption{For large choices of $N$ and $k$ the set $S_N$ is small and the 
 set $E_N$ has empty intersection with $\De_k$.  Then we can concentrate 
 on the set $D_N$, on which the Radon-Nikodym derivative 
 $\nicefrac{\text{d}\nu_2}{\text{d}\nu_1}$ is bounded.}
 \label{fig_admlem3}
\end{figure}

By combining Lemma \ref{lem_admmsrseq} and Proposition \ref{prop_rxyz}  we have the following corollary.

\begin{cor}\label{cor_doesnotdepend}
 Fix a nonnegative integer $\mf\in\Z_{\geq0}$, a vector $\vec{k}\in\Z^\mf$,
 any two linearly summable sequences $\bn$ and $\bnp$, and two admissible 
 measures $\nu$ and $\nu'$.  Then the metric spaces 
 $(\mmk,d^{\nu,\bn}_{\mf,[\vec{k}]})$ and $(\mmk,d^{\nu',\bnp}_{\mf,[\vec{k}]})$
 have the same topology generated by their respective metrics.
\end{cor}

\subsection{\texorpdfstring{$d$}{The function} is a metric}

While it does not hold in general that the minimum of even a finite 
collection of metrics will be itself a metric, it does hold in this
particular case. For this section fix an admissible measure $\nu$, 
a linear summable sequence $\bn$, a nonnegative integer $\mf$, 
and $\vec{k},\vec{k}'\in\Z^\mf$.  Let $\Dplain$ denote $\D$ and 
let $\Dpplain$ denote $\Dp$, as given in Definition~\ref{def_metric}.
It is clear that $\Dplain$ is positive 
definite and it is symmetric because $\permkkp$ is closed under inverses 
so we must only show that the triangle inequality holds.  We show 
this in Lemma~\ref{lem_triangleineq_d} but first we must prove two lemmas.
\begin{lemma}\label{lem_permkkpp}
 Fix $\vec{k},\vec{k}',\vec{k}''\in\Z^\mf$ and let $\permkkp$ be as in 
 Definition~\ref{def_permkkp}.  Then for any fixed $q\in \permkkpp$ 
 we have that $\permkppkp = \{p\circ q^{-1} \vert p\in \permkkp\}$.
\end{lemma}

\begin{proof}
 Let $r\in\permkppkp$.  Then there exist constants $c_1,c_2\in\Z$ such 
 that 
 \[
  k_j-k_{q(j)}'' = c_1 \text{ and } k_j'' - k_{r(j)}' = c_2
 \]
 for all $j=1,\ldots,\mf$. In particular, for $i = q(j)$ we have
% \[
%  k_{q^{-1}(i)} = c_1 + k_i'' \text{ and } k_{r(i)}'=k_i''-c_2.
% \]
% Now using all of this information we can see that 
\[
  k_j-k_{r(q(j))}' = (c_1+k_i'')-(k_i''-c_2) = c_1+c_2
\]
 and so we conclude that $p=r\circ q\in\permkkp$ and clearly 
 $r=p\circ q^{-1}$ so $\permkppkp \subset \{p\circ q^{-1} \vert p\in \permkkp\}$.
 
 Now let $p\in\permkkp$ and $q\in\permkkpp$ so there are constants 
 $c,c_1\in\Z$ such that 
 \[
  k_j -k_{p(j)}' = c \text{ and } k_j -k_{q(j)}'' = c_1.
 \]
 Subtracting these two equations gives $k_{q(j)}''-k_{p(j)}' = c-c_1$
 and thus $p\circ q^{-1}\in\permkppkp$.
\end{proof}

\begin{lemma}\label{lem_dpdqdpq}
 Let $\mele,\melep,\melepp\in\mmk$ and suppose $p\in\permkkp$ and 
 $q\in\permkkpp.$  Then 
 \[
  \Dpplain (\mele,\melep)\leq \Dqplain (\mele,\melepp) + \Dpqplain (\melepp,\melep).
 \]
\end{lemma}

\begin{proof}
 The $\mf=0$ case is trivial so assume $\mf>0$.
 Since $p\in\permkkp$ and $q\in\permkkpp$ there must be constants 
 $c,c_1\in\Z$ such that 
 \[
  k_j -k_{p(j)}' = c \text{ and } k_j -k_{q(j)}'' = c_1.
 \]
 Because $\Dpplain$ is a sum of distances we can use the triangle 
 inequality for each term with an appropriate permutation on the elements:
 \begin{align*}
  \Dpplain (\mele,\melep) =& \sum_{\vec{u}\in\{0,1\}^\mf}\nu\big( t_{\vec{u}}(\De) \symdiff t_{p(\vec{u})}(T^{-c}(\De')) \big)\\ &+ \sum_{j=1}^\mf \Big( \dtsz((S_j)^\infty, (S_{p(j)}')^\infty) + \big|h_j - h_{p(j)}'\big|\Big)\\
                 \leq& \sum_{\vec{u}\in\{0,1\}^\mf}\Big[\nu\big( t_{\vec{u}}(\De) \symdiff t_{q(\vec{u})}(T^{-c_1}(\De'')) \big)+\nu\big( t_{q(\vec{u})}(T^{-c_1}(\De'')) \symdiff t_{p(\vec{u})}(T^{-c}(\De')) \big)\Big] \\&+ \sum_{j=1}^\mf \Big( \dtsz((S_j)^\infty, (S_{q(j)}'')^\infty) + \dtsz((S_{q(j)}'')^\infty, (S_{p(j)}')^\infty) \\
                 &+ \big|h_j - h_{q(j)}''\big|+\big|h_{q(j)}'' - h_{p(j)}'\big|\Big) \\
                 =& \Dqplain (\mele,\melepp) + \Dpqplain (\melepp,\melep).
 \end{align*}

\end{proof}

Notice that in the case that $p=q=\text{Id}$ this gives a proof of the 
triangle inequality for $\Didplain$.

\begin{lemma}\label{lem_triangleineq_d}
 The triangle inequality holds for $\Dplain$.
\end{lemma}

\begin{proof}
 Let $\mele,\melep,\melepp\in\mmk$. There exists some $q\in\permkkpp$ 
 such that $\Dplain (\mele, \melepp) = \Dqplain (\mele, \melepp)$ and 
 by Lemma~\ref{lem_permkkpp} we know that 
 \[
 \min_{p\in\permkkp}\{ \Dpqplain (\melepp, \melep)\}=\Dplain (\melepp,\melep).
 \]
 Now, using the inequality from Lemma \ref{lem_dpdqdpq} we have that
 \begin{align*}
  \Dplain (\mele, \melep) &= \min_{p\in\permkkp} \{ \Dpplain (\mele, \melep)\}\\
                 &\leq \min_{p\in\permkkp} \{ \Dqplain (\mele, \melepp) + \Dpqplain (\melepp, \melep)\}\\
                 &= \Dqplain (\mele, \melepp) + \min_{p\in\permkkp}\{ \Dpqplain (\melepp, \melep) \}\\
                 &= \Dplain (\mele, \melepp) + \Dplain (\melepp, \melep)
 \end{align*}
as desired.
\end{proof}

Combining the arguments in Sections~\ref{sec_metricts} and~\ref{subsec_poly} with the present section, 
in particular Proposition~\ref{prop_rxyz} and Lemma~\ref{lem_triangleineq_d}, we get the following.

\begin{prop}\label{prop_metric}
 Let $\mf\in\Z_{\geq0}$, $\vec{k}\in\Z^\mf$, $\bn$ be a linear summable 
 sequence, and $\nu$ an admissible measure.  Then the space 
 $\metricspacemmk$ is a metric space.
\end{prop}

\subsection{Relation to the metric on the moduli space of toric systems}\label{sec_toric}

In~\cite{PePRS2013} Pelayo-Pires-Ratiu-Sabatini construct a metric on the moduli space of 
(compact) toric integrable systems which we denote by $\toric$.  
Recall there is a one-to-one correspondence between elements of $\toric$ and 
Delzant polytopes.  The authors
of~\cite{PePRS2013} define a metric on $\toric$ by pulling back the natural 
metric on the space of Delzant polytopes given by the Lebesgue measure of 
the symmetric difference.

Toric integrable systems can also be viewed as compact semitoric systems 
with no focus-focus singularities.  If $\mf=0$ then $\Gmf\times\tk=\varnothing$
and thus the affine invariant is a unique polygon, the Delzant polytope.
To compare two such systems the semitoric metric defined in the present paper 
takes the $\nu$-measure of the symmetric difference of the polygons for some 
admissible measure $\nu$, as opposed to using the standard Lebesgue measure 
on $\R^2$ as is done in~\cite{PePRS2013}. Notice also that $\toric$ is not 
equal to $\semitoricz$ because, for instance, there are elements of
$\semitoricz$ which are not compact.

Moreover it is possible for two toric systems to be isomorphic
as semitoric systems but not isomorphic as toric systems.  
This is because if $(M,\om, (J,H))$ and $(M',\om', (J',H'))$
are two choices of 4 dimensional toric systems then a 
diffeomorphism $\phi:M\to M'$ is an isomorphism of toric
systems if $\phi^* (J',H') = (J, H)$.  This corresponds to
taking $f$ to be the identity in the definition of semitoric
isomorphisms.  Thus we see that
if $\sim$ represents the equivalence induced by semitoric isomorphisms 
we have that $\toric / {\sim} \subset \semitoricz$ so the metric on 
$\toric$ produces a topology on a subset of $\semitoricz$ via the quotient topology.

In $\semitoricz$ the semitoric invariant is a unique polygon so to conclude that 
the metrics produce the same topology it is sufficient to show that the same 
sequences of convex compact polygons converge with respect to both the Lebesgue
measure and any admissible measure.

\begin{lemma}\label{lem_compactcase}
 Let $\De_k,\De\subset\R^2$ be convex compact sets for each $k\in\N$,
 let $\mu$ denote the Lebesgue measure on $\R^2$, and let $\nu$ be 
 any admissible measure.  Then  
 $\lim_{k\to\infty}\mu (\De\symdiff\De_k)=0$ if and only 
 if $\lim_{k\to\infty} \nu(\De\symdiff\De_k)=0$.
\end{lemma}

\begin{proof}
 If $\lim_{k\to\infty}\mu (\De\symdiff\De_k)=0$ we can see that 
 $\lim_{k\to\infty}\nu_0 (\De\symdiff\De_k)=0$ where $\nu_0$ is 
 the example of an admissible measure from Section~\ref{subsec_poly}.
 This is because $\nu_0(A)<\mu(A)$ for any set $A\subset\R^2$.  
 Thus we conclude that $\lim_{k\to\infty}\nu (\De\symdiff\De_k)=0$ by
 Lemma~\ref{lem_admmsrseq}.
 
 Now we will show the other direction. Suppose 
 $\lim_{k\to\infty}\nu (\De\symdiff\De_k)=0$ and fix $\varepsilon>0$. 
 Choose some $L>0$ such that $\pi_1(\De)\subset [-L,L]$. By 
 Lemma~\ref{lem_convexinterval} we know there exists $x_0,y_0,y_1\in\R$ with 
 $y_0<y_1$ and $x_0\in[-L,L]$ such that the set 
 $\{x_0\}\times[y_0,y_1]\subset\De$ is a subset of $\De_k$ for 
 $k>K_1$ for some fixed $K_1\in\N$. Now, suppose that $k>K_1$ and $p\in\De_k$ 
 has $\pi_1(p)>L+1$.  Then, since $\De_k$ is convex, the triangle with vertices
 $(x_0,y_0), (x_0,y_1), p$, which we will denote by $G_p$, must be a subset of
 $\De_k$.  Since $\pi_1(\De)\subset[-L,L]$ we know that
 $G_p\setminus \pi_1^{-1}([-L,L])\subset\De\symdiff\De_k$ and the $\nu$-measure 
 of any such triangle $G_p$ defined by a point $p\in\R^2$ with $\pi_1(p)>L$ is 
 bounded below by a constant $c_1=\nu(G_{p_0})>0$ where $p_0=(L+1,0)$. This is
 because any triangle $G_p$ where $\pi_1(p)>L$ contains a triangle $G_{(L+1,y)}$
 for some $y\in\R$ and any such triangle is the image under a vertical, and thus
 $\nu$-preserving, transformation of $G_{p_0}$. Similarly, $p\in\De_k$ for $k>K_1$
 with $\pi_1(p)<-L$ would imply that $\nu(\De\symdiff\De_k)>c_2$ for some constant
 $c_2>0$.  Thus, since $\lim_{k\to\infty}\nu (\De\symdiff\De_k)=0$ we conclude that
 there exists some $K_2>K_1$ such that $k>K_2$ implies that 
 $\De_k \subset \pi_1^{-1} ([-L,L])$. Since $\nu$ is admissible we know that there 
 exists some $c_3>0$ such that $\nicefrac{\text{d}\mu}{\text{d}\nu}<c_3$ on 
 $\pi_1^{-1} ([-L,L])$.  Choose $K_3>K_2$ such that $k>K_3$ implies that 
 $\nu (\De\symdiff\De_k)<\nicefrac{\varepsilon}{c_3}$ and notice that
 \[
  \mu(\De\symdiff\De_k) = \frac{\text{d}\mu}{\text{d}\nu} \nu(\De\symdiff\De_k) < c_3 \nu(\De\symdiff\De_k)<\varepsilon,
 \]
because while the Radon-Nikodym derivative is not bounded on all of 
$\R^2$ it is bounded on the set $\De\symdiff\De_k$ for large enough $k$.
\end{proof}

\begin{cor}\label{cor_toric}
 The metric $\Dplain$ induces the same topology on $\toric$ as 
 the metric defined in~\cite{PePRS2013} does.
\end{cor}

 Corollary~\ref{cor_toric} follows from Lemma~\ref{lem_compactcase}.
 This result is concerning compact polygons.
 Of course, if we consider non-compact sets these 
 metrics will not induce the same topology.

\begin{remark}\label{rmk_cpt}
Let $\semitoriccpt\subset\semitoric$ be the collection of \emph{compact} 
semitoric integrable systems.  Then the polygons produced will always be
compact and thus Lemma \ref{lem_compactcase} applies.  So we can conclude
that when restricting to $\semitoriccpt$ the standard Lebesgue measure can
be used in place of the choice of admissible measure and the same topology
will be produced.
\end{remark}

\subsection{\texorpdfstring{$d$ and $d^{\mathrm{Id}}$}{The metrics} induce the same topology}\label{sec_topology} 

Let 
\[
 \mmknoperm=\{\mele\in\mmk\mid\mele\text{ is in twisting index class }\vec{k}\}
\]
and define $\Didplain = \Didnomk$ on $\m$ by 
\[
 \Didplain(\mele, \melep) = \left\{\begin{array}{ll}\Did(\mele,\melep)&\text{ if }\mele,\melep\in\mmknoperm\text{ for some }\mf\in\Z_{\geq0},\vec{k}\in\Z^\mf\\\infty&\text{ otherwise}.\end{array}\right.
\]
The function $\Didplain$ compares the focus-focus points of systems in
the order of increasing $J$, so systems with twisting indexes which
can be compared only by reordering (i.e.~those which are in the same generalized twisting index
class but not in the same twisting index class) will be assigned a value
of $\infty$ by $\Didplain$.
Both $\Dplain$ and $\Didplain$ are defined on $\m$ and the main result of 
this section will be that both of these metrics induce the same topology
on $\m$.

\newcommand{\meler}{\mele_r}
\newcommand{\Dpnplain}{d^{p_n}}
\newcommand{\Dprplain}{d^{p_r}}
\begin{lemma}\label{lem_laconvg}
 Let $\mele,\meler\in\m$ for $r\in\N$.  
 Then $\D(\mele, \meler) \stackrel{r\to\infty}{\longrightarrow} 0$
 implies that $\la_j^r \stackrel{r\to\infty}{\longrightarrow} \la_j$
 for all $j=1,\ldots,\mf$.
\end{lemma}

\begin{proof}
Again we use $\Dplain$ to denote $\D$ and $\Dpplain$ to denote $\Dp$.

 {\it Step 1}: Let $p_r\in\perm$ satisfy 
 $\Dplain(\mele, \meler) = \Dprplain (\mele,\meler)$ for each $r\in\N$. 
 For the first step of this proof we will argue that 
 $\la_{p_r(j)}^r \stackrel{r\to\infty}{\longrightarrow} \la_j$ by contrapositive.
 Suppose there exists some $j\in1,\ldots,\mf$ such that $\la_{p_r(j)}^r \not\to \la_j$
 as $r\to\infty$.
 This means there exists $a>0$ and a subsequence $(r_i)_{i=0}^\infty$
 such that 
 \[
  \abs{\la_{p_{r_i}(j)}^{r_i}-\la_j} > a \text{ for all } i\in\N.
 \]
 Now let $t_j = t^1_{\ell_{\la_j}}$ and $t^r_j = t^1_{\ell_{\la^r_{p_r(j)}}}$.  
 Let $\De$ be a polygon which represents a choice of 
 $\vec{\varepsilon}=\{+1,\ldots,+1\}$ for $\mele$. We must show that 
 $\nu ( t_j (\De) \symdiff t_j^{r_i} (\De))$ is bounded away from zero.  
 We may assume that $a$ is less than the horizontal distance from $\la_j$ to 
 the edge of the polygon $\De$ because 
 $ \la_j \in \mathrm{int}(\pi_1(\De))$.
 Let $b=\min_{x\in[\la_j -a, \la_j+a]} \{\,\text{length}(\De \cap \ell_x)\,\}$
 and notice that since $\De$ is a convex polygon we must have that $b>0$.
 
 The set $\De$ may be shifted by a vertical transformation so that 
 $\max \{\pi_2 (\De \cap \ell_x)\}=0$ for each $x\in\R$ to form a 
 new set $\De'\subset\R^2$, as is shown in Figure~\ref{fig_laconvg1}.
 Let $A\colon\R^2\to\R^2$ be the composition
 of these transformations so $A(\De) = \De'$. This new set may not
 be convex but since $\nu$ is invariant under vertical 
 translations we have that $\nu(\De) = \nu(\De')$.  Notice that 
 $\mathcal{B}=[\la_j-a,\la_j+a]\times[-b,0]$ satisfies 
 $\mathcal{B}\subset \De'$.
 
 \begin{figure}
 \centering
 \includegraphics[height=100pt]{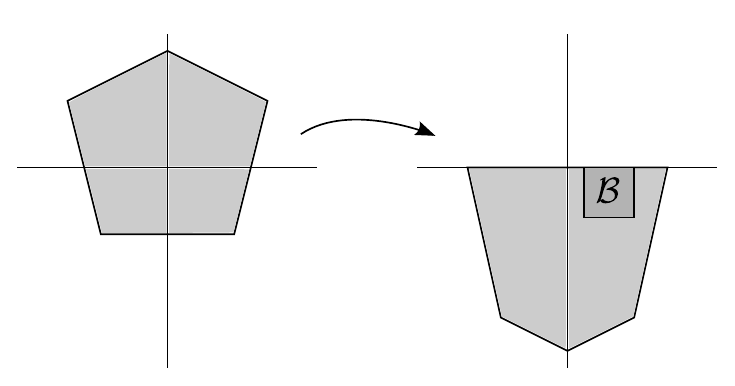}
 \caption{Without changing the $\nu$-measure we can produce a new polygon 
 which has $\{y=0\}$ as its top boundary.}
 \label{fig_laconvg1}
\end{figure}
 
 Now there are two cases, both shown in Figure~\ref{fig_laconvg2}.
 If $\la_j < \la_j^{r_i}$ then 
 $t_j (\mathcal{B})\cap \{y>0\} \subset t_j (\De') \symdiff t_j^{r_i} (\De')$.
 This is because $t_j^{r_i}$ is the identity on points where $x\in[\la_j-a, \la_j+a]$
 and so for $x$ in this interval $\De'$ does not intersect the open upper half plane.
 The set $t_j (\mathcal{B})\cap \{y>0\}$ always contains the rectangle 
 $[\la_j + \nicefrac{a}{2}, \la_j+a]\times [0,\nicefrac{a}{2}]$, 
 as in Figure~\ref{fig_laconvg2}. 
 Let $c_1 = \nu([\la_j + \nicefrac{a}{2}, \la_j+a]\times [0,\nicefrac{a}{2}])$.
 
\begin{figure}
 \centering
 \includegraphics[height=170pt]{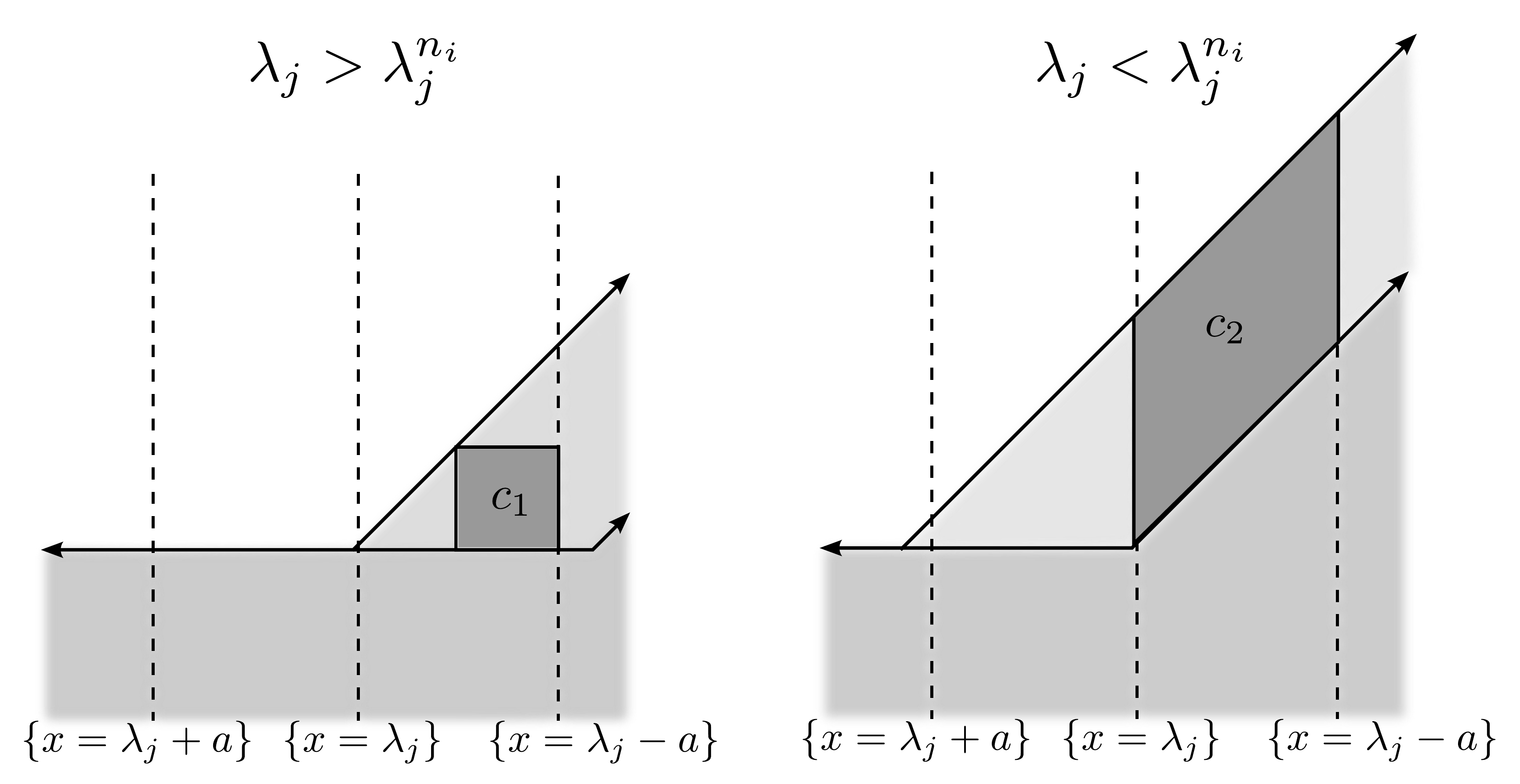}
 \caption{Either $\la_j^{r_i}< \la_j-a$ or $\la_j^{r_i}> \la_j+a$.  
          Each case is shown above and in either case there is some 
          positive measure region which is always in the symmetric 
          difference.  This causes convergence to be impossible.}
 \label{fig_laconvg2}
\end{figure}
 
 Now suppose that $\la_j>\la_j^{r_i}$.  In this case the symmetric 
 difference always contains the region $t_j ([\la_j,\la_j+a]\times[a-b,a])$
 which has the same measure as $[\la_j,\la_j+a]\times[a-b,a]$; see 
 Figure~\ref{fig_laconvg2}. Let $c_2 = \nu([\la_j,\la_j+a]\times[a-b,a])$
 and let $c=\min\{c_1,c_2\}$.  So in any case we have that 
 $\nu(t_j (\De) \symdiff t_j^{n_i}(\De) \geq c > 0$.
 
 Assume that $\lim_{r\to\infty} \Dplain (\mele, \meler)=0$. 
 This implies $\lim_{i\to\infty} \nu(\De \symdiff \De^{r_i})= 0$.
 In this case fix $\varepsilon>0$ such that $\varepsilon<c$, and 
 find $I>0$ such that $i>I$ implies that 
 $\nu(\De \symdiff \De^{r_i})<\varepsilon$.  Then for $i>I$ we have that
 \begin{align*}
  \nu[t_j (\De) \symdiff t_j^{r_i} (\De)] &\leq \nu[t_j(\De) \symdiff t_j^{r_i} (\De^{r_i})]+\nu[t_j^{r_i} (\De^{r_i}) \symdiff t_j^{r_i}(\De)]\\
                                              &= \nu[t_j(\De) \symdiff t_j^{r_i} (\De^{n_i})]+\nu[\De^{r_i} \symdiff \De],
 \end{align*}
which implies
\[
 \nu[t_j(\De) \symdiff t_j^{r_i} (\De^{r_i})] \geq \nu[t_j (\De) \symdiff t_j^{r_i} (\De)] - \nu[\De^{r_i} \symdiff \De] > c-\varepsilon.
\]
Thus $\lim_{r\to\infty}\nu[t_j(\De) \symdiff t_j^{r}(\De^{r})]= 0$ 
is impossible, but this is a term in $\Dplain (\mele, \meler)$ so
$\Dplain (\mele, \meler) \to 0$ is impossible as well. 
We conclude that $\la_{p_r(j)}^r \to \la_j$ for all $j=1,\ldots, \mf$.
 
 {\it Step 2}: From Step 1 we know that 
 $\la_{p_r(j)}^r\stackrel{r\to\infty}{\longrightarrow}\la_j$ for each
 $j=1,\ldots,\mf$. Let 
 $D = \min \{\abs{\la_j - \la_{j'}}\mid j,j'\in\{1,\ldots,\mf\}, j\neq j'\}$.
 Then there exists some $R>0$ such that $r>R$ implies that 
 $\abs{\la^r_{p_r(j)} - \la_j} < \nicefrac{d}{2}$.  Thus, for $r>R$ we have 
 that $p_r = \text{Id}$ and the result follows.
\end{proof}

\begin{prop}\label{prop_sametop}
 Let $\mf\in\Z_{\geq0}$, $\vec{k}\in\Z^\mf$, $\bn$ be a linear summable 
 sequence, and $\nu$ be an admissible measure. Then $\D$ and $\Did$ 
 induce the same topology on $\m$.
\end{prop}

\begin{proof}
 Any sequence which converges for $\Didplain$ will converge 
 for $\Dplain$ because $\Dplain<\Didplain$.
 Suppose that $(\melen)_{n=1}^\infty$ is a sequence in $\m$ which 
 converges to $\mele\in\m$ with respect to $\Dplain$.  Then by 
 Step~2 of the proof of Lemma~\ref{lem_laconvg} we know there 
 exists some $N>0$ such that for $n>N$ we have that 
 $\Dplain(\mele,\melen)=\Didplain(\st,\st_n)$.
 Thus, we see that the sequence $\Didplain(\mele,\melen)$ is eventually 
 equal to a sequence which converges to zero, so we conclude that 
 $\Didplain(\mele,\melen)\stackrel{n\to\infty}{\longrightarrow}0$.
\end{proof}

\begin{remark}\label{rmk_components}
 Even though they can be compared by the metric, 
 each $\mmknoperm\subset\m$ is in a separate component (in terms of connectedness)
 of $(\m,\Dplain)$.  
 This is because these are defined to be in different components for $\Didplain$
 and we have just shown that $\Didplain$ and $\Dplain$ induce the same topology.
\end{remark}

\section{The completion}\label{sec_compl}

In this section we compute the completion of the space of semitoric ingredients $\m$ 
which corresponds to the completion of $\semitoric$ by Theorem \ref{thm_class}.  
We will show that the completion of $\m$ is $\mcompl$, where $\mcompl$ is as is described
in Definition~\ref{def_mcomplmk} and Definition~\ref{def_polycompl}.
The completion of an open interval in $\R$ with the usual metric is the corresponding 
closed interval and we have already stated that $\rxyz$ is complete 
(Proposition~\ref{prop_rxyz}),  so to produce the completion of $\m$ it seems the 
only difficultly will be with the weighted polygons.  This is not the case since 
in fact defining the distance as a minimum of permutations has intertwined the metrics
on these different spaces so we cannot consider them separately.  This section has 
similar arguments to those in~\cite{PePRS2013} except that in our case we must consider
a whole family of polygons all at once instead of only one polygon.  For the remainder 
of this section fix some admissible measure $\nu$, some linear summable sequence $\bn$,
a nonnegative integer $\mf$, and a vector $\vec{k}\in\Z^\mf$.  For simplicity we will 
use $\Dplain$ and $\Dpplain$ to refer to $\D$ and $\Dp$ (from Definition~\ref{def_metric})
respectively, where $p\in\perm$.

In Section~\ref{sec_buildcompl} we show that the completion must contain 
$\mcompl$ and in the remaining subsections we show that $\mcompl$ is 
complete. In Section~\ref{sec_cauchy} we prove several Lemmas about 
Cauchy sequences which are used in Section~\ref{sec_complcompl} to 
conclude that $\mcompl$ is in fact the completion of $\m$.

With the metric presented in the current paper, there is no way for elements of $\m$ with different numbers of focus-focus
points or which are in different generalized twisting index classes to be close to one another because the 
distance between any two such systems is always $\infty$ (see Definition~\ref{def_fullmetric}).
Thus, we will work with the components $\mmk$ of $\m$.

First, notice that the definition of $\Dplain$ from Definition~\ref{def_metric} 
holds on $\mcompl$ as well.  That is, extend the definition of $\Dplain$ in the 
following way:

\begin{definition} 
  Suppose that \[\mele=\stfulla,\,\,\,\melep=\stfullap\in\mcompl.\]
 Then:
 \begin{enumerate}
  \item the \emph{comparison with alignment} $p$ is 
        \[
         \Dpplain (\mele,\melep)= \distpolyp ([\Aw],[\Awp]) + \sum_{j=1}^\mf \Big( \big|h_j - h_{p(j)}'\big|+ \dtsz((S_j)^\infty, (S_{p(j)})^\infty)\Big);
        \]
  \item the \emph{the distance between $\mele$ and $\melep$} is 
        \[
         \Dplain (\mele, \melep) = \min_{p\in \permkkp} \left\{\Dpplain(\mele,\melep) \right\}. 
        \]
 \end{enumerate}
\end{definition}

\begin{prop}
 $\Dplain$ is a metric on $\mcompl$.
\end{prop}

This proposition follows from the proof of Proposition \ref{prop_metric}.

\begin{remark}
 Notice that $\Didplain$ is not a metric on $\mcompl$ because it does 
 not satisfy the triangle inequality.  This can be seen in 
 Example~\ref{ex_needmin}.
\end{remark}

Throughout Section~\ref{sec_buildcompl} each space we examine can be viewed 
as a subspace of $\mcompl$ and we will endow them with the structure of a 
metric subspace.

\begin{remark}\label{rmk_msubsetmcompl}
The space $\m$ can be viewed as a subspace of $\mcompl$ because there is a
natural correspondence between the elements of $\m$ and the elements of a 
subset of $\mcompl$.  This is because there is at most one element of $\m$ 
in each equivalence class in $\mcompl$ so the space $\m$ corresponds to the
subset $\{[\mele]\mid\mele\in\m\}$.
\end{remark}

\subsection{The completion must contain \texorpdfstring{$\mcomplmk$}{BLAH}}\label{sec_buildcompl}

In the next few lemmas we start with $\mmk$ and build up to $\mcomplmk$ 
in several steps, showing that each inclusion is dense.  First we will 
show that the completion of $\mmk$ must include at least all rational 
labeled polygons which satisfy the convexity requirements.
We will use the following result of Pelayo-Pires-Ratiu-Sabatini.
\begin{lemma}[{\cite[Remark 23]{PePRS2013}}]\label{lem_rmk23}
 Any corner of a rational convex polygon can be edited in a small
 neighborhood so that it is still a rational convex polygon and 
 in that neighborhood every corner is Delzant.
 Moreover, such a neighborhood can be made as small as desired.
\end{lemma}

\begin{lemma}\label{lem_apoly}
 Let $\apoly\subset\polycompl$ be given by 
 \[
  \apoly=\left\{ [\lwppluskp] \left| \begin{array}{l} t_{\vec{u}}(\De)\in\poly\text{ for any }\vec{u}\in\{0,1\}^\mf,\\ \vec{k}\sim\vec{k}',\nu(\De)<\infty,\textrm{ and}\\  \min_{s\in\De}\pi_1(s)<\la_1<\ldots<\la_\mf<\min_{s\in\De}\pi_1(s) \end{array} \right. \right\}
 \]
 and let 
 \[
  \mmka = \apoly \times [0,1]^\mf \times \rxyz^\mf.
 \]
 Then the inclusion $\mmk\subset\mmka$ is dense.
\end{lemma}

\begin{proof}

 Fix any element $\mele = \stfull \in \mmka$.  Since 
 $\Dplain\leq\Didplain$ we will show there exists an element $\melep\in\mmk$ 
 arbitrarily close to $\mele$ with respect to the function $\Didplain$.  
 Clearly we will have no problems with making the volume invariant or the
 Taylor series arbitrarily close so just consider the polygons.

 Let $[\dew]=[\lwpplus]\in\apoly$ and fix $\varepsilon>0$.  We will show 
 there exists some element $[\dewp]\in \dpoly$ such that 
 $\distpolyid([\dew],[\dewp])<\varepsilon$.  We will choose this element of 
 $\dpoly$ to have the same $\la_j$ values as $\dew$.  Since the action of 
 $t_{\vec{u}}$, $\vec{u}\in\{0,1\}^\mf$, does not change the volume of sets
 we have
 \[
  \distpolyidplain([\dew], [\dewp]) \leq 2^\mf \nu (\De \symdiff \De')
 \]
 where $\lwppplus\in[\dewp].$  To complete the proof it suffices 
 to show that there exists an element $\lwppplus\in\dpoly$ such that $\De$ 
 and $\De'$ are equal except on a set of $\nu$-measure less than 
 $2^{-\mf}\varepsilon$.

 For $j=1,\ldots,\mf$ let $p_j\in\R^2$ be the intersection of $\ell_{\la_j}$ 
 with the top boundary of $\De$.  Let $U\subset\R$ be a union of disjoint 
 neighborhoods around each corner of $\De$ which is not an element of 
 $\{p_j\}_{j=1}^\mf$ such that $\nu (U) < \nicefrac{\varepsilon}{2^{\mf+1}}$. 
 Also, let $V\subset\R\setminus U$ be a union of disjoint neighborhoods around
 each point $p_j$ for each $j=1,\ldots,\mf$ and 
 $\nu (V) < \nicefrac{\varepsilon}{2^{\mf+1}}$.  We will define $\De'$ in several 
 stages, editing it several times.  Start by assuming that $\De'=\De$. 
 By Lemma~\ref{lem_rmk23} we can edit $\De'$ on the set $U$ so that
 every vertex is Delzant except possibly the ones in $V$.  
 
 Now, recall that for a \emph{semitoric} polygon to be Delzant the points 
 $p_j$ must all either be fake or hidden Delzant corners.  This is equivalent
 to saying that the corners on the top boundary of $t_{\vec{u}}(\De')$ must 
 all be Delzant for $\vec{u} = <1,\ldots,1>$.  Since $t_{\vec{u}}(\De')$ is
 a convex polygon and $t_{\vec{u}}(V)$ is a neighborhood of the edges 
 $t_{\vec{u}}(p_j)$ we can again use Lemma~\ref{lem_rmk23} to conclude
 that we may edit $t_{\vec{u}}(\De')$ inside of the set $V$ such that all 
 of the vertices on the top boundary are Delzant.  Now we have finished 
 defining $t_{\vec{u}}(\De')$ and since this map is invertible we have also
 defined $\De'$.  Notice that for $j=1,\ldots,\mf$ each point $t_{\vec{u}}(p_j)$ 
 is either a Delzant corner, which would make $p_j$ a hidden Delzant corner,
 or it is not a vertex at all, in which case $p_j$ would be a fake corner.  
 Also, it is easy to check that any new Delzant corner we had to define in 
 $t_{\vec{u}}(V)$ which is not on the point $t_{\vec{u}}(p_j)$ for some 
 $j=1,\ldots,\mf$ gets transformed by $t_{\vec{u}}^{-1}$ to form a Delzant
 corner on $\De'$.  In conclusion, $[\dewp]$ is a Delzant semitoric polygon 
 and each of the $2^\mf$ polygons in the equivalence class is equal to each
 polygon in $[\dew]$ except on a set of $\nu$-measure less than 
 $\nicefrac{\varepsilon}{2^\mf}$.
\end{proof}

So from the above Lemma we conclude that the completion of $\mmk$ must contain $\mmka$.
In the next Lemma we show it must contain a larger set. The only difference between 
$\apoly$ and $\bpoly$ is that $\bpoly$ allows irrational polygons.

\begin{lemma}\label{lem_bpoly}
 Let 
 \[
  \bpoly=\left\{ [\lwppluskp] \left| \begin{array}{l} t_{\vec{u}}(\De)\text{ is a convex polygon for any }\vec{u}\in\{0,1\}^\mf,\\ 0<\nu(\De)<\infty, \vec{k}\sim\vec{k}'\textrm{ and}\\  \min_{s\in\De}\pi_1(s)<\la_1<\ldots<\la_\mf<\max_{s\in\De}\pi_1(s) \end{array} \right. \right\}
 \]
 and let 
 \[
  \mmkb = \bpoly \times [0,1]^\mf \times \rxyz^\mf.
 \]
 Then the inclusion $\mmka\subset\mmkb$ is dense.
\end{lemma}

\begin{proof}
 
 Just as in the proof of Lemma \ref{lem_apoly} we can see that we only need 
 to consider the polygons.  Suppose that $[\dew]\in\bpoly$ and $\lwpplus\in[\dew]$.
 Given any $\varepsilon>0$ we can find an open neighborhood of the boundary of 
 $\De$ which has $\nu$-measure less than $\varepsilon$ (since the boundary has 
 measure zero and $\nu$ is regular) and we may approximate $\De$ by a rational 
 polygon with boundary inside of this neighborhood. In the case that $\De$ is 
 compact this can be done by approximating the irrational slopes with rational
 ones (exactly as done in~\cite{PePRS2013}).
 
 This strategy will work even if $\De$ is not compact. For the faces of $\De$ 
 which are non-compact with irrational slope (if there are any) we can still 
 approximate these with a line of rational slope because of the properties of
 the admissible measure $\nu$.  Suppose there is a non-compact face of $\De$
 which has irrational slope $r\in\R\setminus\Q$.  Then choose $q\in\Q$ such 
 that $q<r$ and $\nu(\{qx<y<rx\})<\varepsilon$ and let the edge on the rational
 polygon have slope $q$.  Such a slope can be chosen because if the measure of
 that set is always finite and replacing $q$ by $q_2=\nicefrac{q+r}{2}$ will 
 produce a wedge with half the measure of the original.
 
\end{proof}

\begin{remark}\label{rmk_order}
 Recall that for simple semitoric systems we order the focus-focus
 points by their $J$-value, that is, we use the $x$-component of the
 location of the momentum map image of the focus-focus point.
 Since in the completion it is possible for $\la_j = \la_{j+1}$ for some $j\in 1,\ldots, \mf-1$ the 
 order in which the critical points are labeled in a system cannot be made unique 
 by only considering the $x$-components. 
 This means that there could be two elements in $\mcomplmk$ which have the same 
 invariants except labeled in a different order.  Of course, we do not want this
 because these two elements should be the same, so we use the other invariants 
 to create a unique ordering on the critical points of any element of $\mcomplmk$.
 We fix the order so that if $\la_j=\la_{j+1}$ for some $j=1,\ldots,\mf-1$ then 
 we require that $h_j\leq h_{j+1}$.  In the case that $\la_j = \la_{j+1}$ and 
 $h_j = h_{j+1}$ we look to the Taylor series.  In this situation we require that
 the coefficient of $X$ of the Taylor series $(S_j)^\infty$ is less than or equal
 to the coefficient of $X$ in $(S_{j+1})^\infty$ and if those are equal we look 
 to the coefficient of $Y$ and continue in this fashion.  Now given any system
 with critical points there is a unique order in which to label them which is
 essentially the lexicographic order on the invariants.
\end{remark}

For the next Lemma we only slightly change the restrictions on the 
$(\la_j)_{j=1}^\mf$.  Notice that we allow $\la_i\leq \la_{i+1}$ instead
of $\la_i<\la_{j+1}$ and additionally allow (positive only) infinite
values for the $\la_j$.  This can only happen in the case that the 
polygon is non-compact.  If $\la_j=+\infty$ then we define $t_{j}^1$
to be the identity because all of $\R^2$ is to the left of this value.
We allow positive infinity and not negative infinity because as
$\la_j\to\infty$ the map $t^1_j$ becomes identity but as
$\la_j\to-\infty$ the map $t^1_j$ does not converge to anything.

\begin{lemma}\label{lem_mmkc}
 Let 
 \[
  \cpoly=\left\{ [\lwppluskp] \left| \begin{array}{l} t_{\vec{u}}(\De)\text{ is a convex polygon for any }\vec{u}\in\{0,1\}^\mf,\\ 0<\nu(\De)<\infty, \vec{k}\sim\vec{k}',\\ \la_j\in\R\cup\{\infty\} \text{ for }j=1,\ldots,\mf\text{, and }\\ \min_{s\in\De}\pi_1(s)\leq\la_1\leq\ldots\leq\la_\mf \leq\max_{s\in\De}\pi_1(s) \end{array} \right. \right\}\]and let \[\mmkc = \cpoly \times [0,1]^\mf \times \rxyz^\mf.
 \]
 The inclusion $\mmkb\subset\mmkc$ is dense.
\end{lemma}

\begin{proof}
 Again, we only need to consider the polygons.  We will prove this Lemma in 
 two steps.  First, suppose that $[\dew]\in\cpoly$ has $\la_j<\infty$ for 
 each $j=1,\ldots,\mf$ so the only thing that is keeping $[\dew]$ from 
 being in $\bpoly$ is the possibility that $\la_j=\la_{j+1}$ for some 
 fixed $j\in\{1,\ldots,\mf-1\}$. Let $\vec{u}$ be all zeros except for a
 1 in the $j^{th}$ and $(j+1)^{st}$ positions.  Then $[\dew]\in\cpoly$ 
 implies that $t_{\vec{u}}(\De)$ is convex so we know that there is a 
 vertex of $\De$ on the top boundary with $x$-coordinate $\la_j$.  Let 
 $m_1$ denote the slope of the edge to the left of this vertex and let 
 $m_2$ denote the slope to the right.  Then we can see that the convexity
 of $t_{\vec{u}}(\De)$ implies that $m_1\geq m_2+2$.  Now we want to show
 that there exists some $[\dewp]\in\bpoly$ arbitrarily close in $\Didplain$
 to $[\dew]$.  Let $[\dewp]$ be equal to $[\dew]$ except that 
 $\la_j'<\la_j<\la_{j+1}'$ and that the top boundary of $\De'$ has slope 
 $m_1-1$ on the interval $x\in(\la_j',\la_{j+1}')$.  So, as is shown in 
 Figure~\ref{fig_mmkc1}, we have cut the corner off of $\De$ to produce 
 $\De'$ and clearly this cut can be made as small as desired.  This process
 can be repeated for each instance of $\la_j = \la_{j+1}$ for 
 $j\in\{1,\ldots,\mf\}$.
 
\begin{figure}
 \centering
 \includegraphics[height=170pt]{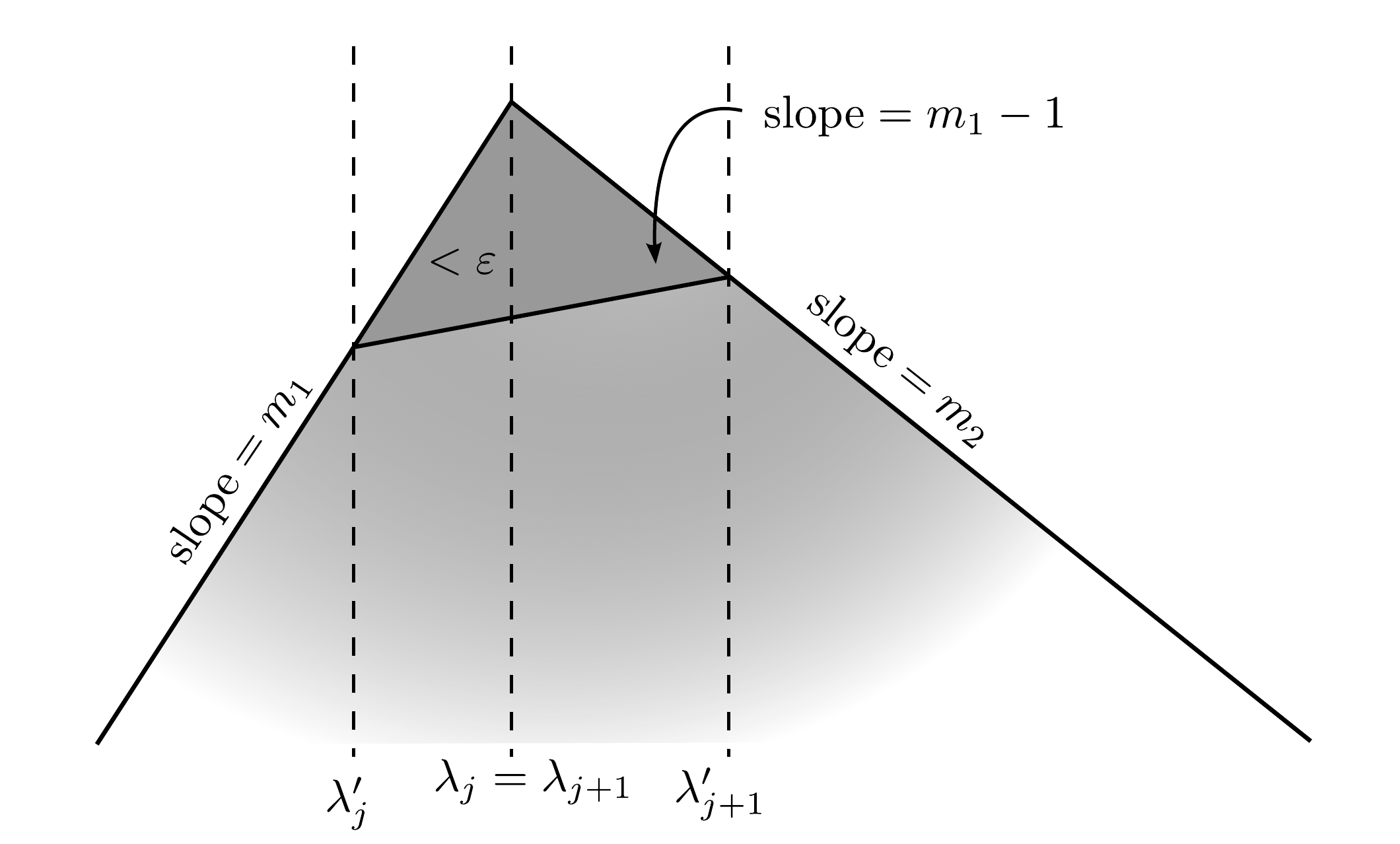}
 \caption{By cutting the corner and adjusting the values of $\la_j$ and 
          $\la_{j+1}$ of an element in $\cpoly$ we can produce an element 
          of $\bpoly$ which is very close.}
 \label{fig_mmkc1}
\end{figure}
 
 Now we proceed to step two. Assume that $[\dew]=[\lwpplus]\in\polycompl$ 
 has $\la_\mf = +\infty$ (and $\la_j<\infty$ for $j=1,\ldots,\mf-1$) and
 we will construct a sequence with $[\dew]$ as its limit.  Let 
 $N = \max_{j=1,\ldots,\mf-1}\abs{\la_j}$ and for any $n\in\N$ 
 which satisfies $n>N$ define a set $\De^n=\De\cap[-n,n]$ with 
 $\la_\mf = n$.  That is 
 \[
  [\De_w^n] = (\De^n, (\ell_{\la_j}, +1)_{j=1}^{\mf-1}, (\ell_{n}, +1)).
 \]
 Notice that each polygon in each family $[\De^n_w]$ is convex because 
 it is the intersection of two convex sets. Then 
 $\distpoly([\De_w],[\De_w^n])\to 0$.  Clearly a similar process can 
 be used to produce sets which have multiple $\la$ values which are 
 infinite.
  
\end{proof}

Next we would like to consider arbitrary convex sets, but there is a subtlety.
So far we have only been working with polygons
and if the symmetric difference of two polygons has zero measure in $\nu$,
and therefore also in $\mu$, those polygons are the same set, but
this is not true for arbitrary subsets of $\R^2$.
For the measure of the symmetric difference to produce a metric on the collection
of subsets of $\R^2$ one must only consider these sets up to measure zero corrections.
Thus, instead of considering only convex sets we will now consider all sets which are convex up to measure zero corrections
(as is done in~\cite{PePRS2013}).
Recall that $\nu$ and the Lebesgue measure $\mu$ have precisely the same measure zero sets,
so the equivalence relation in the following definition does not depend on
the choice of admissible measure.

\begin{definition}\label{def_polycompl}
 Let 
 \[
  \mathcal{C}_{\mf,[\vec{k}]}=\left\{ [\lwpplusakp] \left| \begin{array}{l} A\subset\R^2, \la_j\in\R\cup\{\infty\}\text{ for }j=1,\ldots,\mf,\\t_{\vec{u}}(A)\text{ is a convex set for any }\vec{u}\in\{0,1\}^\mf,\\  \vec{k}\sim\vec{k}', 0<\nu(\De)<\infty \text{, and}\\  \min_{s\in A}\pi_1(s)\leq\la_1\leq\ldots\leq\la_\mf\leq\max_{s\in A}\pi_1(s) \end{array} \right. \right\}.
 \]
 Further, for any measurable sets $A,B\subset\R^2$ we say $A\simeq B$ if 
 and only if $\mu(A\symdiff B)=0$ and let $[A]$ denote the equivalence 
 class of $A$ with respect to this relation.  Finally, let 
 \[
  \polycompl =\left\{\big[\lwpplusaeq\big] \left| \begin{array}{l}[\lwpplusa]\in\mathcal{C}_{\mf,[\vec{k}]} \text{ or}\\\nu(A)=0\text{ and }\la_j=0\text{ for }j=1,\ldots,\mf\end{array}\right.\right\}.
 \]
\end{definition}

Here it is important to notice that we have included one extra element 
in each $\polycompl$, the equivalence class of the empty set.  For this
element the values of $\la_j$ are unimportant so we set them all equal
to zero (in fact, any fixed number will work).
For the last Lemma in this section we will show that the inclusion in 
$\mcomplmk$, which is defined in Definition \ref{def_mcomplmk}, is 
also dense.
The explanation of 
how $\mmkb$ can be viewed as a subspace of $\mcomplmk$ is in 
Remark~\ref{rmk_msubsetmcompl}.

\begin{lemma}\label{lem_cpoly}
 The inclusion $\mmkc\subset\mcomplmk$ is dense.
\end{lemma}

\begin{proof}
 
  Once more, we only have to consider the labeled weighted convex sets 
  since it is easy to align the volume invariant and Taylor series 
  invariant. Let $[\lwpplusa]=[\dew]\in\polycompl$.  Now pick 
  $[\lwpplusb]\in\cpoly$ and notice that they have the same 
  $\la$ values, so if $A$ and $B$ are close then so are all of
  the other polygons.  Simply approximate $A$ by a family of 
  disjoint rectangles contained in $A$.  We need to be sure 
  that $t_{\vec{u}}(B)$ is convex for any choice of $\vec{u}\in\{0,1\}^\mf$
  so take $B$ to be the convex hull of the rectangles which approximate 
  $A$ from the inside and the points in the top boundary of $A$ which 
  have $x$-value equal to $\la_j$ for some $j\in\{1,\ldots,\mf\}$. 
  Since $B\subset A$ and $t_{\vec{u}}(A)$ is convex around $x=\la_j$ 
  for each $j=1,\ldots,\mf$ we know that $t_{\vec{u}}(B)$ is convex 
  (Figure \ref{fig_cpoly}).
  
\begin{figure}
 \centering
 \includegraphics[height=150pt]{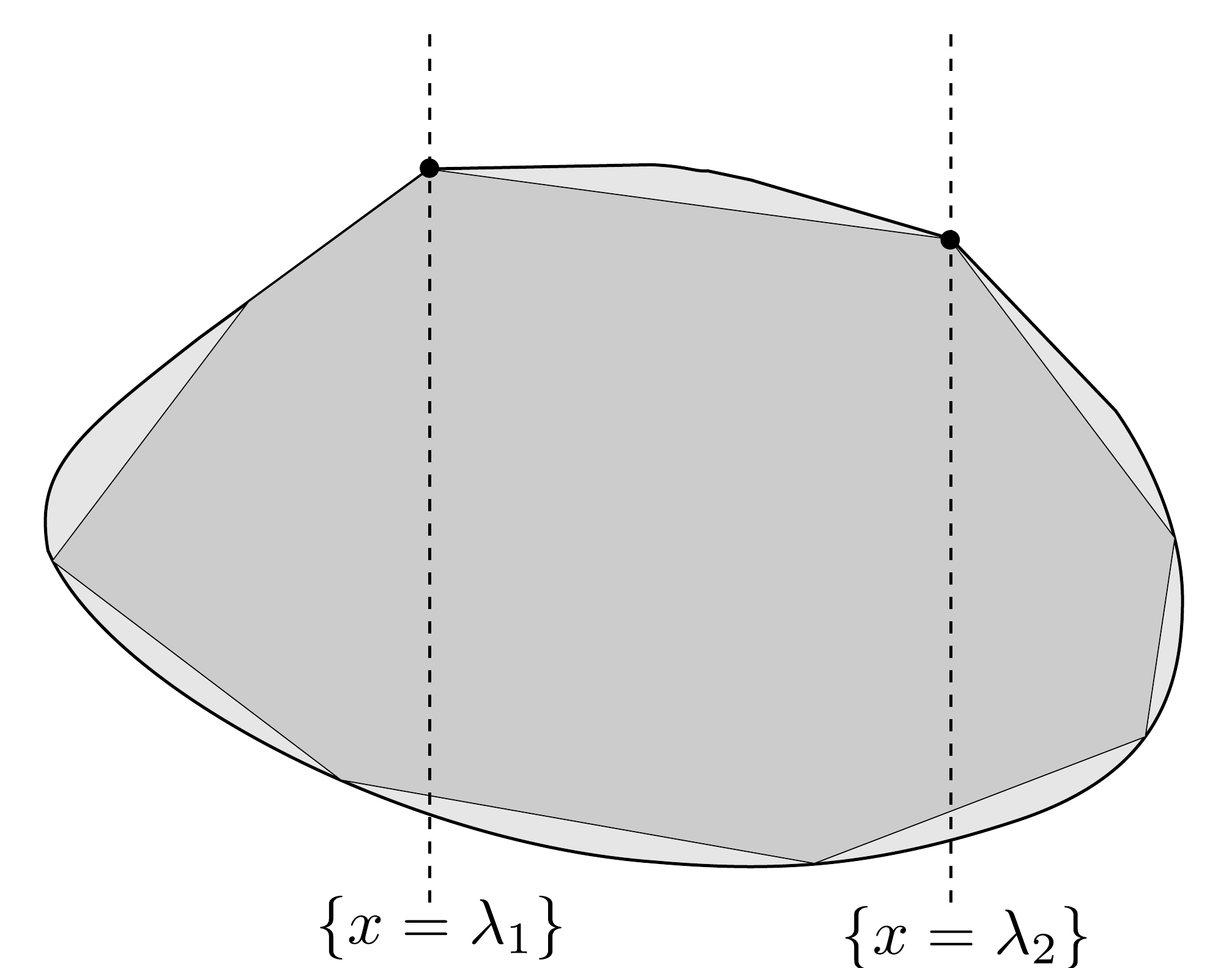}
 \caption{An arbitrary convex set can be approximated from the inside by a polygon.  The convexity requirements will be met as long as the vertices on the top boundary at $\{x=\la_j\}$ for each $j=1,\ldots,\mf$ are included in the polygon.}
 \label{fig_cpoly}
\end{figure}
 
\end{proof}

From the results of Lemma~\ref{lem_apoly}, Lemma~\ref{lem_bpoly}, 
Lemma~\ref{lem_mmkc}, and Lemma~\ref{lem_cpoly}, the following lemma is immediate.
\begin{lemma}\label{lem_mdense}
 The completion of $\mmk$ must contain $\mcomplmk$.
\end{lemma}

In this section we have undergone several extensions of $\m$ to obtain $\mcompl$.
We see that the limits of elements of $\m$ are associated to sets which are convex
(up to measure zero corrections) and which are not required to be rational polygons.
It is possible for the $x$-components of the positions of the images of the focus-focus
points, which are usually required to be distinct, to be equal in the limit.  Also, it is possible
for the positions of the images of the focus-focus points, which must be in the interior of
the moment map image for semitoric systems, to limit to any point on the boundary.

\subsection{Cauchy sequences for \texorpdfstring{$\Dplain$ and $\Didplain$}{the different metrics}}\label{sec_cauchy}

In this section we investigate the relationship between Cauchy sequences in $\Didplain$ and $\Dpplain$.
This will be used to prove Lemma~\ref{lem_mcompliscompl}; that $\mcompl$ is complete.

\begin{lemma}\label{lem_cauchysub}
 Let $\melen\in\mcomplmk$ for $n=1,\ldots,\infty$. If $(\melen)_{n=1}^\infty$ is Cauchy with respect to $\Dplain$ then there exists a subsequence $(\mele_{n_i})_{i=1}^\infty$ which is Cauchy with respect to $\Didplain$.
\end{lemma}

\begin{proof}
 Let $(\melen)_{n=1}^\infty$ be as in the statement of the Lemma.  Let $A_0=\N$ and let $M_0 = 0$.  We will define $A_n$ and $M_n$ recursively for each $n\in\N$.  Suppose that $\abs{A_{n-1}}=\infty$ and $M_{n-1}\in A_{n-1}$. Let $\varepsilon_n = 2^{-n}$. Find some $M>0$ such that $k,l>M$ implies that $\Dplain(\mele_k,\mele_l)>\nicefrac{\varepsilon_n}{2}$. Now let $M_n$ be any element of $A_{n-1}$ which is greater than $M$ and $M_{n-1}$.  This means $\Dplain (\mele_{M_n},\mele_l) <\nicefrac{\varepsilon_n}{2}$ for any $l>M_n$.  For $p\in\perm$ let $\mathcal{B}_p^n = \{ l\in A_{n-1}\mid l>M_n, \Dpplain (\mele_{M_n}, \mele_l)< \nicefrac{\varepsilon_n}{2}\}.$ Notice that $\cap_{p\in\perm} \mathcal{B}_p^n = \{l\in A_{n-1} \mid l>M_n\}$ by the definition of $M_n$.  The union of this finite number of sets has infinite cardinality so at least one of those sets must also have infinite cardinality.  Choose any $p_n\in\perm$ such that $\abs{\mathcal{B}_{p_n}^n}=\infty$ (there may be several possible choices). Now define $A_n = (A_{n-1} \cap [0,M_n])\cup \mathcal{B}_{p_n}^n.$  Notice that $M_n\in A_n$ and $\abs{A_n}=\infty$.
 
 Now let $A= \cap_{n\in\N}A_n$ and notice that $\abs{A}=\infty$ because $\{M_n\mid n\in\N\}\subset A$.
 
 So $(\mele_{a})_{a\in A}$ is a subsequence of $(\mele_n)_{n=1}^\infty$.  We will show that this subsequence is Cauchy with respect to $\Didplain$.  Fix any $\varepsilon>0$ and find $n\in\N$ such that $\varepsilon_n<\varepsilon$. Now pick any $k,l>M_n$ with $k,l\in A$.  Then $k,l\in A_n$ implies that $k,l \in \mathcal{S}_{p_n}^n$ so $\Dpnplain (\mele_{M_n}, \mele_k), \Dpnplain (\mele_{M_n}, \mele_l)<\nicefrac{\varepsilon_n}{2}<\nicefrac{\varepsilon}{2}$.  Also notice that $p_n$ being an appropriate permutation to compare $\mele_k$ with $\mele_{M_n}$ and also appropriate to compare $\mele_l$ with $\mele_{M_n}$ implies that $\text{Id}\in\perm$ is an appropriate permutation to compare $\mele_k$ and $\mele_l$. Thus \[\Didplain (\mele_k, \mele_l) \leq \Dpnplain (\mele_k, \mele_{M_n}) + \Dpnplain (\mele_l, \mele_{M_n}) < \varepsilon\] by Lemma \ref{lem_dpdqdpq}.
\end{proof}

\begin{lemma}\label{lem_cauchyidconvg}
 Suppose that $(\melen)_{n=1}^\infty$ is a sequence of elements of $\mcomplmk$ which is Cauchy with respect to the function $\Didplain$.  Then there exists some $p\in\perm$ and $m\in\mcomplmk$ such that \[\lim_{n\to\infty} \Dpplain (m_n,m) = 0.\]
\end{lemma}

\begin{proof}
 For $A,B\subset\R^2$ say $A\simeq B$ if and only if $\nu(A\symdiff B) = 0$ and let $\mathcal{F}$ denote the subsets of $\R^2$ with finite $\nu$-measure modulo $\simeq$. Now let $\mathcal{E}=\{[A]\in\mathcal{F} \mid \text{there exists } B\in[A]\text{ which is convex}\}$ and let $d_\mathcal{E}$ be the metric on this space given by the $\nu$-measure of the symmetric difference.  For simplicity we will write $A\in\mathcal{E}$ instead of $[A]\in\mathcal{E}$. We will show that this metric space is complete. Let $\chi_A$ denote the characteristic function of the set $A\in\mathcal{E}$.  Then for $A,B\in\mathcal{E}$ we can see that \[d_\mathcal{E} (A,B) = \int_{\R^2} \abs{\chi_A - \chi_B}d\nu = \norm{\chi_A - \chi_B}_{L^1}\] the $L^1$ norm on $(\R^2, \nu).$ Now suppose that $(A^k)_{k=1}^\infty$ is a Cauchy sequence in $(\mathcal{E},d_\mathcal{E})$ and by measure zero adjustments we can assume that each $A^k$ is convex. Then $(\chi_{A^k})_{k=1}^\infty$ is Cauchy in $L^1(\R^2, \nu)$ and thus there must exist some function $g:\R^2\to\R$ defined up to measure zero such that \[\lim_{k\to\infty} \norm{g-\chi_{A^k}}_{L^1} = 0\] because $L^1$ is complete. 
 
 The functions $(\chi_{A^k})_{k=1}^\infty$ converge to $g$ in $L^1$ 
 so we know that there is a subsequence $(\chi_{A^{k_n}})_{n=1}^\infty$
 which converges to $g$ pointwise off of some measure zero set $S$.  Let
 \[
  A = \{x\in\R^2\setminus S \mid g(x)=1\}
 \]
 and now we will show that $A$ is almost everywhere equal to a convex set so $\mathcal{E}$ is complete. Let $A'$ be the convex hull of $A$ and we will show that $\nu(A\symdiff A')=0$.  Let $p\in A'$ which means there exists $q,r\in A$ and $t\in [0,1]$ such that $p=(1-t)q+tr$.  Since the subsequence $(\chi_{A^{k_n}})_{n=1}^\infty$ converges pointwise to $\chi_A$ at the points $q$ and $r$ (since $q,r\in A$ and $A$ is disjoint from $S$) this means that there exists some $N>0$ such that $n>N$ implies $q,r\in A^{k_n}$.  Thus, since each $A^k$ is convex we see that for $n>N$ we have $p\in A^{k_n}$.  We conclude that $p\in A \cup S$ and thus $A\symdiff A'\subset S$ so $\nu(A\symdiff A')=0$.  Also notice $\nu(A^k, A) \to 0$ as $k\to\infty$ implies that $\nu(A)<\infty$. This means $A\in\mathcal{E}$ so $(\mathcal{E}, d_\mathcal{E})$ is a complete metric space.
 
 Let $([A_w^k])_{k=1}^\infty$ be a Cauchy sequence in $(\polycompl,\distpolyid)$. Let \[ [A_w^k] = [(A^k,(\ell_{\la_j^k}, +1, k_j)_{j=1}^\mf)] \text{ and let } \akep = t_{\vec{u}}^k (A^k) \] for each $\vec{\varepsilon}\in\{-1,1\}^\mf$ with $u_j = \frac{1-\ep_j}{2}$.  Since this sequence is Cauchy we also know that the sequence $(\akep)_{k=1}^\infty$ is a Cauchy sequence in $(\mathcal{E},d_\mathcal{E}).$  Thus for each $\vec{\varepsilon}\in\{-1,1\}^\mf$ there exists some convex $A_{\vec{\ep}}\in\mathcal{E}$ which is the limit of $(\akep)_{k=1}^\infty$ in $\mathcal{E}$. Let $A=A_{(1,\ldots,1)}$. We have produced a family of convex, $\nu$-finite sets which could be the limit, but we still need to check that there is some choice of $(\La_j)_{j=1}^\mf$ such that $A_{\vec{\ep}}=t_{\vec{u}}(A_0)$ in $\mathcal{E}$ for each $j=1,\ldots,\mf$.
 
 Fix some $j\in\{1,\ldots,\mf\}$ and let $A_j = A_{\ep}$ where $\ep_j=-1$ and $\ep_i=+1$ for $i\neq j$ and let $t_k$ denote $t^{1}_{\la_j^k}$.  Since $\nu$ is invariant under vertical translations we have that  \[ d_\mathcal{E} (t_k (A), t_{\vec{u}}(A^k)) = d_\mathcal{E} (A, A^k)\] so both go to zero as $k\to\infty$.  By the triangle inequality we can see that \[ d_\mathcal{E} (t_k(A), A_j) \leq d_\mathcal{E} (t_k(A), t_k(A^k))+ d_\mathcal{E} (t_k(A^k), A_j)\] so we conclude that
 \begin{equation}\label{eqn_dEtozero}
  d_\mathcal{E} (t_k(A), A_j)\to 0 \text{ as } k\to\infty.
 \end{equation}
 
 If $(\la_j^k)_{k=1}^\infty$ diverges to $+\infty$ or converges to $\sup (\pi_1 (A))$ then we are done.  This is because in this case $d_\mathcal{E}(t^k(A), A_j)\to 0$ as $k\to\infty$ implies that $A$ and $A_j$ represent the same element in $\mathcal{E}$ (i.e. they are equal almost everywhere) and $t^\La$ acts as the identity on $A_0$ if $\La$ is the rightmost value of $A_0$.
 
 \begin{figure}
 \centering
 \includegraphics[height=210pt]{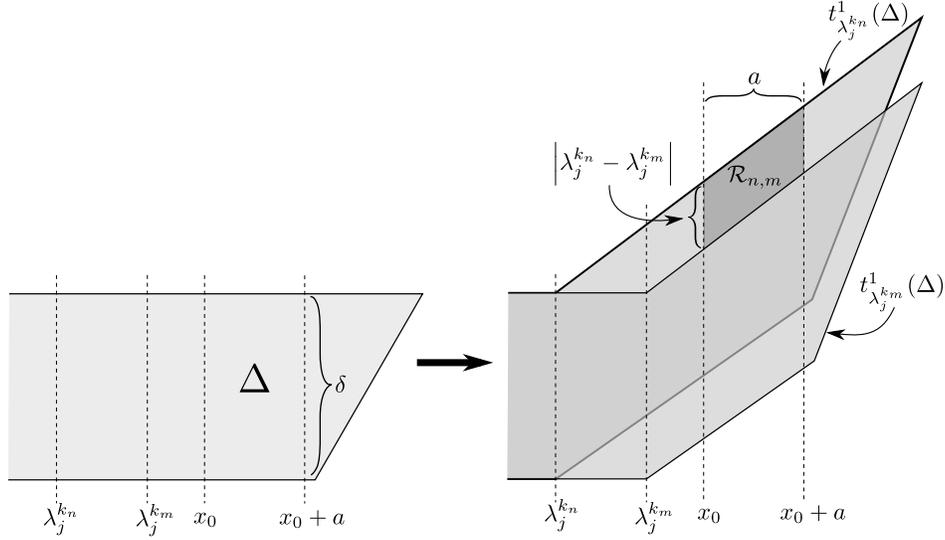}
 \caption{The action of $t^1_{\la_j^{k_m}}$ and $t^1_{\la_j^{k_n}}$ on a polygon. It can be seen that $\mathcal{R}_{n,m}$ is a subset of the symmetric difference and has measure which is nonzero if $\abs{\la_j^{k_n}-\la_j^{k_m}}\neq 0$.}
 \label{fig_cauchyidconvg1}
\end{figure}
 
Otherwise we can find some $x_0, a\in\R$ with $a>0$ such that 
$[x_0, x_0+2a] \subset \pi_1(A)$ and there exists a subsequence
$(\la_j^{k_n})_{n=1}^\infty$ such that $\la_j^{k_n}<x_0$ for all $n$.
Notice that $A\cap \ell_{x}$ is an interval for any $x\in\pi_1(A)$ 
because $A$ is convex.  Let $\de_1 = \text{length}(A \cap \ell_{x_0})$
and $\de_2 = \text{length}(A \cap \ell_{x_0+a})$ and notice that 
$\de_1, \de_2<\infty$ because otherwise we would have $\nu(A)=\infty$
or $\nu(A)=0$ because $A$ is convex and $\nu$ is invariant under 
vertical translations.  Also notice that 
$\text{length}(A \cap \ell_x) \geq \min\{\de_1, \de_2\}$ for any 
$x\in[x_0,x_0+a]$ because $A$ is convex.  Pick any $n,m\in\N$ and
we can see that $t^1_{\la_j^{k_n}}$ and $t^1_{\la_j^{k_m}}$ only 
differ by a vertical translation when acting on 
$A \cap \pi_1^{-1} ([x_0,x_0+a])$ 
(see Figure \ref{fig_cauchyidconvg1}).  This guarantees that there 
is a region $\mathcal{R}_{n,m}$ in the symmetric difference 
$t^{\la_j^{k_n}}(A)\symdiff t^{\la_j^{k_m}}(A)$ which has the same
measure as a rectangle of length $a$ and height 
$\min \{\de_1, \de_2, \abs{\la_j^{k_n}-\la_j^{k_m}}\}$ positioned 
between the $x$-values of $x_0$ and $x_0+a$ (since $\nu$ is 
translation invariant). 
If $R$ stands for the measure of a rectangle from $x=x_0$ 
to $x=x_0+a$ of unit height we can see that 
\[
 \nu(\mathcal{R}_{n,m})=\min\left\{\de_1, \de_2, \abs{\la_j^{k_n}-\la_j^{k_m}}\right\} R
\]
and since $\mathcal{R}_{n,m}\subset t^{\la_j^{k_n}}(A)\symdiff t^{\la_j^{k_m}(A)}$ 
we know that
 \begin{equation}\label{eqn_rectmeas}
 \min\{\de_1, \de_2, \abs{\la_j^{k_n}-\la_j^{k_m}}\} R\leq \nu(t_{\la_j^{k_n}}(A)\symdiff t_{\la_j^{k_m}}(A)).
 \end{equation}
 The right side of Equation \eqref{eqn_rectmeas} is Cauchy with respect to $m$ and $n$ 
 because $(t_{\la_j^{k_n}})_{n=1}^\infty$ converges by Equation \eqref{eqn_dEtozero} and thus 
 the left side is Cauchy as well.  This means that $(\la_j^{k_n})_{n=1}^\infty$ is a Cauchy sequence of real numbers and 
 thus must converge.  Call its limit $\La_j \in\R$.  To complete the proof we must 
 only show that $\nu(t^1_{\La_j} (A) \symdiff A_j) = 0$.  This is clear because 
 \[\nu(t^1_{\La_j} (A) \symdiff A_j) \leq \nu(t^1_{\La_j} (A) \symdiff t^1_{\la_j^{k_n}}(A)) + \nu( t^1_{\la_j^{k_n}}(A) \symdiff A_j)\] 
 and the right side goes to zero as $n\to\infty$.  So we conclude that the original Cauchy sequence 
 converges to $ [(A, (\ell_{\La_j}, +1, k_j)_{j=1}^\mf)]$.  Clearly the elements of each copy 
 of $\rxyz$ and $[0,1]$ can be made to converge.  The only problem is that 
 possibly this limit does not have the critical points labeled in the correct order 
 according to Remark \ref{rmk_order} to be an element of $\polycompl$ so we reorder it by some permutation $p\in\perm$ and the result follows.
\end{proof}

\subsection{\texorpdfstring{$\mcompl$}{BLAH} is complete}\label{sec_complcompl}

\begin{lemma}\label{lem_mcompliscompl}
 $\mcomplmk$ is complete.
\end{lemma}

\begin{proof}  Any Cauchy sequence in $\mcomplmk$ must have a subsequence which is Cauchy with respect to $\Didplain$ 
by Lemma \ref{lem_cauchysub}.  By Lemma \ref{lem_cauchyidconvg} that sequence must converge with respect to $\Dpplain$ 
for some fixed $p\in\perm$, which in particular means that it must converge with respect to $\Dplain$.  A Cauchy 
sequence with a subsequence which converges must converge.
\end{proof}

Now Lemma \ref{lem_mdense} and Lemma \ref{lem_mcompliscompl} imply the main result of this section.

\begin{prop}\label{prop_complete}
 Given an admissible measure $\nu$ and a linear summable sequence the completion of $(\m,\Dnomk)$ is $(\mcompl,\Dnomk)$.
\end{prop}

The reason to use $\Dplain$ instead of $\Didplain$ can be seen by the 
examining structure of the completion, as can be seen
in the following example.
\begin{example}\label{ex_needmin}
 Let 
\[
 [\De_w^l]=\left\{\begin{array}{ll} \,[(\De_l, (\la_1=0,\ep_1=1,k_1=0), (\la_2=l,\ep_2=1,k_2=0))]&\text{ if }l>0\\ \,[(\De_l, (\la_1=l,\ep_1=1,k_1=0), (\la_2=0,\ep_2=1,k_2=0))]&\text{ if }l<0\end{array}\right.
\]
and suppose that $m_l\in\m$ is a system given by 
\[
 m_l = \left\{\begin{array}{ll}([\De_w^l],( (S_1)^\infty, h_1),( (S_2)^\infty, h_2)) & \text{ if }l>0\\([\De_w^l],( (S_2)^\infty, h_2),( (S_1)^\infty, h_1)) & \text{ if }l<0\end{array}\right.
\]
for $l\in[-1,1]\setminus\{0\}$ such that $\lim_{l\to0}m_l$ exists in 
$(\mcompl,\Dplain)$. This can be thought of as one of the critical points 
being fixed and the other passing over it at $l=0$ as is shown in 
Figure~\ref{fig_needmin}.  The complications in defining this come from the 
fact that the order of the critical points switches at $l=0$ so the labeling
has to switch. Now we can see the problem with using $\Didplain$:
the metric should reflect the fact that these systems are approaching the same
limiting system, so we should have $\lim_{l\to0^+}m_l=\lim_{l\to0^-}m_l$
(which is true in the topology induced by $\Dplain$), but $\lim_{l\to0^+}m_l\neq\lim_{l\to0^-}m_l$ with respect to $\Didplain$.

\begin{figure}[hb!]
 \centering
 \includegraphics[height=120pt]{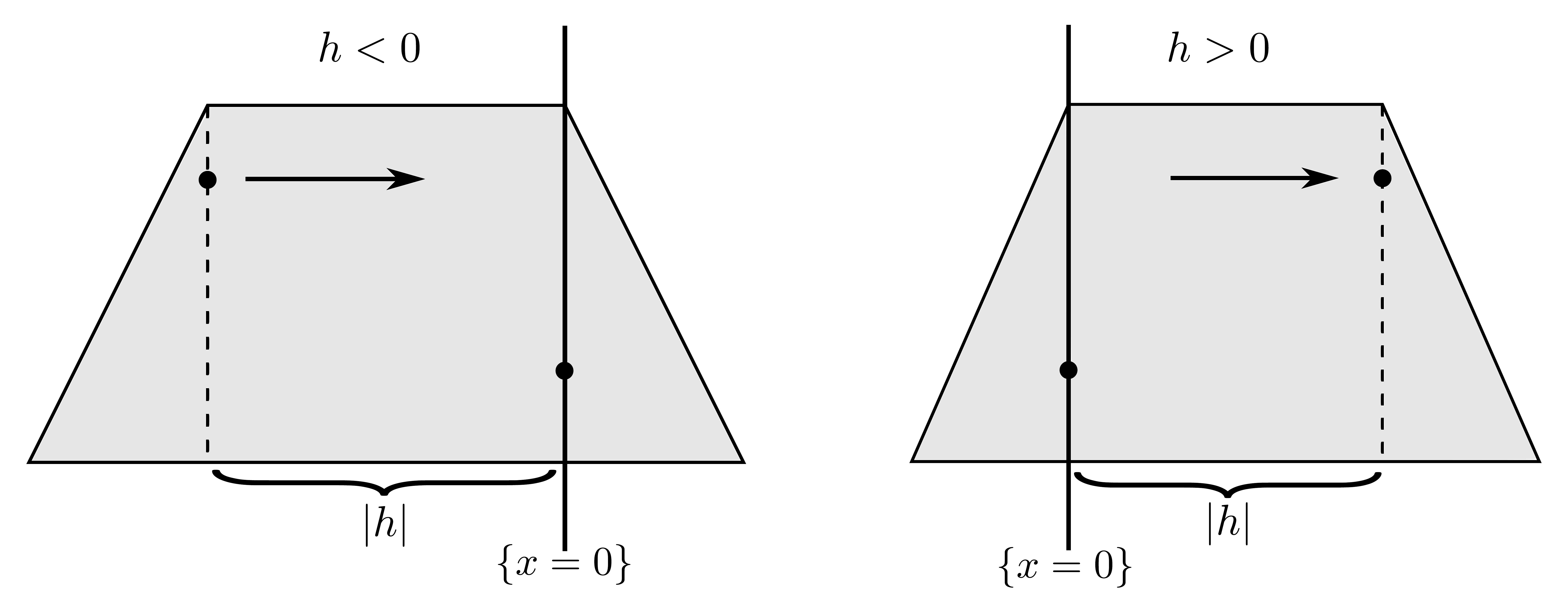}
 \caption{A continuous family in $\mcompl$ in which one critical point 
          passes over the other as $h$ increases from negative to positive.}
 \label{fig_needmin}
\end{figure}

\end{example}

Finally, Theorem \ref{thm_main} is produced by combining Proposition \ref{prop_metric},
Corollary \ref{cor_doesnotdepend}, and Proposition \ref{prop_complete}.

\section{Further questions}\label{sec_quest}

Now that we have defined a metric, and in particular a topology, on 
$\semitoric$ there are several questions that would be natural to 
address.  First of all, one may be interested extending the metric
defined in this paper in the way that this paper has extended the 
metric from~\cite{PePRS2013}.  To produce such an extension to a 
larger class of integrable systems one would first have to classify
those systems with invariants in a way which extends the Pelayo-V\~{u} Ng\d{o}c classification
from~\cite{PeVNsemitoricinvt2009, PeVNconstruct2011}.  
Also, one can now ask what are the connected components of $\semitoric$.
Furthermore, with a topology on $\semitoric$ we can consider 
Problem~2.45 from~\cite{PeVNfirst2012}, which asks what the closure of
the set of semitoric integrable systems would be when considered as a 
subset of $C^\infty (M, \R^2)$.  To address this problem an appropriate 
topology on $C^\infty (M, \R^2)$ would have to be defined.  This situation
is much more general than the systems which are the focus of this paper 
so it may be best to study metrics constructed in a more general case such
as in~\cite{Pa2015}.

This paper is partially motivated by the desire to understand limits of semitoric
systems which are themselves not semitoric.  One method to do this is to 
study the elements of $\mcompl\setminus\m$ in relation to integrable systems. 
Perhaps some subset of this can be interpreted as 
corresponding to non-simple semitoric systems or to some other type of integrable 
system not included in the classification by 
Pelayo-V\~{u} Ng\d{o}c~\cite{PeVNsemitoricinvt2009, PeVNconstruct2011}. Problem~2.44 
from~\cite{PeVNfirst2012} asks if some integrable systems may be expressed as the 
limit of semitoric systems in an appropriate topology and the study of 
$\mcompl\setminus\m$ may make some progress on this question.

The topology on toric systems~\cite{PePRS2013} allows
Figalli-Pelayo in~\cite{FPe2014} to explore the continuity 
properties of the maximal toric ball packing density function,
$\Om\colon \toric \to [0,1]$, which assigns to each toric system the portion
of the manifold which can be filled by disjoint equivariantly embedded balls.
Now that a topology has been defined on $\semitoric$ questions regarding the 
continuity of functions on $\semitoric$ may be asked.  For instance,
one could attempt to define and study a \emph{maximal semitoric ball
packing density function}, $\Om_{\mathcal{T}}:\semitoric\to[0,1]$,
analogous to the toric case.
The function $\Om_{\mathcal{T}}$ would assign to each semitoric system
the portion of the total volume of the 
manifold which may be filled by disjointly embedded symplectic balls, which
are required to embed in a way that respects the semitoric structure of the manifold.  
To study this, one would have to first determine in what way an embedded ball
should respect the structure of semitoric system.
Once this function is defined, the topology produced in this paper could be
used to study its continuity.

{\it Acknowledgements.}  The author is grateful to his advisor {\'A}lvaro 
Pelayo for proposing the question addressed in this article and for 
providing help and advice on many different occasions.  He is also grateful 
to the anonymous referee who supplied many helpful comments and
for the support of the National Science Foundation under agreements 
DMS-1055897 and DMS-1518420.

\bibliographystyle{amsplain}
\bibliography{biblio}

{\small
  \noindent
  \\
  {\bf Joseph Palmer} \\
  University of California, San Diego\\
  Department of Mathematics\\
  9500 Gilman Drive \#0112\\
  La Jolla, CA 92093-0112, USA.\\
  {\em E\--mail}: \texttt{j5palmer@ucsd.edu} \\
}   

\end{document}